\documentclass[a4paper,12pt,reqno]{amsart}
\usepackage[T1]{fontenc}
\usepackage[utf8]{inputenc}
\usepackage{anyfontsize}
\usepackage[english]{babel}
\usepackage{stmaryrd}
\usepackage{amsthm,amsfonts,amsmath,amssymb}
\usepackage{verbatim}

\usepackage[%
	left=3.5cm,       
	right=3.5cm,      
	top=3.5cm,        
	bottom=3.5cm,     
	heightrounded,    
	bindingoffset=0mm 
]{geometry}

\usepackage[dvipsnames]{xcolor}
\usepackage{graphicx}

\usepackage{xfrac}
\usepackage{functan}
\usepackage{mathtools}
\usepackage{mathrsfs}

\numberwithin{equation}{section}
\allowdisplaybreaks

\newcommand{\R}{\mathbb{R}}
\newcommand{\Rn}{\mathbb{R}^N}

\newcommand{\N}{\mathbb{N}}
\renewcommand{\H}{\mathcal{H}}

\newcommand{\di}{\,\mathrm{d}}

\newcommand{\Om}{\Omega}

\DeclareMathOperator{\dist}{dist}

\DeclareMathOperator{\supp}{supp}

\DeclareMathOperator{\esssup}{ess\,sup}
\DeclareMathOperator{\Div}{div}
\renewcommand{\div}{\Div}
\DeclareMathOperator{\Per}{Per}

\usepackage{enumitem}

\usepackage[
colorlinks=true,
urlcolor=teal,
citecolor=magenta,
linkcolor=teal,
linktocpage,
pdfpagelabels,
bookmarksnumbered,
bookmarksopen,
breaklinks=true]
{hyperref}

\usepackage[capitalize,nameinlink,noabbrev]{cleveref}

\theoremstyle{plain}
	\newtheorem{theorem}{Theorem}[section]
    \newtheorem{lemma}[theorem]{Lemma}
    \AddToHook{env/lemma/begin}{\crefalias{theorem}{lemma}}
    \crefname{lemma}{Lemma}{Lemmas}
	\newtheorem{proposition}[theorem]{Proposition}
    \AddToHook{env/proposition/begin}{\crefalias{theorem}{proposition}}
    \crefname{proposition}{Proposition}{Propositions}
	\newtheorem{corollary}[theorem]{Corollary}
    \AddToHook{env/corollary/begin}{\crefalias{theorem}{corollary}}
    \crefname{corollary}{Corollary}{Corollaries}

\theoremstyle{definition}
	\newtheorem{definition}[theorem]{Definition}
    \AddToHook{env/definition/begin}{\crefalias{theorem}{definition}}
    \crefname{definition}{Definition}{Definitions}

    \AddToHook{env/example/begin}{\crefalias{theorem}{example}}
    \crefname{example}{Example}{Examples}
	\newtheorem{remark}[theorem]{Remark}
    \AddToHook{env/remark/begin}{\crefalias{theorem}{remark}}
    \crefname{remark}{Remark}{Remarks}

\title[Spectral inequalities for weighted $p$-Laplacians]{Spectral inequalities for weighted $p$-Laplacians via Talenti symmetrization}%

\author[G.~Bartoli]{Giulio Bartoli}
\address[G.~Bartoli]{
Institut f\"ur Mathematik, Goethe-Universit\"at Frankfurt, \linebreak Robert-Mayer-Str.~10, 60325 Frankfurt am Main, Germany  
}
\email{bartoli@mathematik.uni-frankfurt.de}

\author[G.~Saracco]{Giorgio Saracco}
\address[G.~Saracco]{
Dipartimento di Matematica e Informatica, Universit\`a di Ferrara, via Machiavelli 30, 44121 Ferrara (FE), Italy
}
\email{giorgio.saracco@unife.it}


\begin{document}


\begin{abstract}
We consider the weighted $p$-Laplacian associated with a measure $\mu$ having a continuous positive density with respect to the Lebesgue measure on an open convex subset $X\subset\mathbb{R}^N$. 
We prove that Talenti's weighted P\'olya--Szeg\H{o} inequality---originally established for Lipschitz functions on $X$---extends to the weighted Sobolev space of functions vanishing on the relative boundary of an open set $\Omega\subset X$, defined as the closure of Lipschitz functions supported away from it.
This yields Faber--Krahn-type inequalities for the first $(p,q)$-eigenvalue of the weighted Dirichlet $p$-Laplacian.
%
%
We present several examples fitting this abstract framework, in which the isoperimetric sets are balls, half-spaces, or complements of balls.
\end{abstract}

\maketitle


\section{Introduction}

Let $\mathcal{L}^N$ be the standard $N$-dimensional Lebesgue measure on $\mathbb{R}^N$ and let $X\subset\mathbb{R}^N$ be an open and connected set.
Given a continuous function $W:X\to\mathbb{R}$, we consider the strictly positive measure
\begin{equation}
\label{eq:measures}
\mu = e^{W(x)}\mathcal{L}^N.
\end{equation}
This measure is mutually absolutely continuous with respect to the Lebesgue measure and its Radon--Nikodym derivative is $e^{W}$, which we shall call the \emph{density} or \emph{weight}, while we shall refer to $W$ as the \emph{potential}. %
Measures of this kind are Borel and encompass, for instance, log-convex and log-concave measures that arise in connection with Witten Laplacians and Gaussian-type diffusions.

\subsection*{The functional framework}
For an open set $\Omega\subset X$ and $p\in[1,+\infty)$, we define the space $L^p_\mu(\Omega)$ as the collection of all $\mu$-measurable functions $u$ on $\Omega$ such that
\[
\|u\|_{p,\mu,\Omega}^p = \int_\Omega |u(x)|^p \,\mathrm{d}\mu < +\infty.
\]
This space, equipped with the norm $\|\,\cdot\,\|_{p,\mu,\Omega}$, is a Banach space. %
Further, if \(p\in(1,+\infty)\), by Clarkson's inequality it is uniformly convex, hence reflexive by the Milman--Pettis Theorem. %
From now on, unless otherwise stated, we assume \(p\in(1,+\infty)\).

Since \(W\) is continuous, the weight \(e^W\) is bounded from above and from below by positive constants on every compact subset of \(X\).
In particular,
\begin{equation}
\label{eq:kufner-opic}
e^{-W/(p-1)}\in L^1_{\mathrm{loc}}(X),
\end{equation}
so that, by H\"older's inequality, \(L^p_\mu(\Omega)\subset L^1_{\mathrm{loc}}(\Omega)\), and distributional derivatives of functions in \(L^p_\mu(\Omega)\) make sense. We define the weighted Sobolev space \(W^{1,p}_\mu(\Omega)\) as the space of functions \(u\in L^p_\mu(\Omega)\) whose distributional derivatives \(\partial_i u\) belong to \(L^p_\mu(\Omega)\). 
Thanks to~\eqref{eq:kufner-opic}, see~\cite{KO84}, this space is a Banach space when endowed with the norm
\begin{equation}
\label{eq:norm}
\|u\|_{W^{1,p}_\mu(\Omega)}
=
\bigl(\|u\|_{p,\mu,\Omega}^p + \|\nabla u\|_{p,\mu,\Omega}^p \bigr)^{\frac1p},
\end{equation}
where \(\|\nabla u\|_{p,\mu,\Omega}=\bigl\||\nabla u|\bigr\|_{p,\mu,\Omega}\) and \(|\cdot|\) denotes the Euclidean norm on \(\mathbb{R}^N\).
Moreover, it is reflexive. 
Indeed, since all norms on \(\mathbb{R}^N\) are equivalent, the map \(u\mapsto(u,\partial_1u,\dots,\partial_Nu)\) is an isomorphism of \(W^{1,p}_\mu(\Omega)\) onto a closed subspace of the reflexive space \(L^p_\mu(\Omega)\times [L^p_\mu(\Omega)]^N\), endowed with the \(\ell^p\)-product norm, and reflexivity is invariant under isomorphisms.

We denote by \(\mathrm{Lip}_0(\Omega;X)\) the class of Lipschitz functions \(u\) on \(X\) such that
\begin{itemize}
\item[(i)] 
\(\supp_X(u)\subset\Omega \) is bounded;
\item[(ii)]
\(\dist(\supp_X(u); \partial\Omega\cap X)>0\);
\item[(iii)]
\(\mu(\{\,u\neq0\,\})<+\infty\);
\end{itemize}
where \(\supp_X(u) = \overline{\{\,x\in X\,:\, u(x)\neq0\,\}}^X\) and \(\dist(\,\cdot\, ; \,\cdot\,)\) is the distance on \(\mathbb{R}^N\). 
Here and in the following, \(\overline{\Omega}^A\) denotes the closure of \(\Omega\) relative to the set \(A\), while \(\overline{\Omega}\) denotes its closure in \(\mathbb{R}^N\).
The last condition is automatically satisfied if \(\mu(\Omega)<+\infty\), which is the case in our main results, or if \(\mu\) is finite on bounded subsets of \(X\), that is, \(\mu(X\cap B_R)<+\infty\) for all \(R>0\), as in all the examples of \cref{sec:examples} but one. 
We stress that the latter does not follow from the continuity of \(W\), since the weight may blow up at \(\partial X\), as it happens in \cref{sec:complements}. 

Every \(u\in\mathrm{Lip}_0(\Omega;X)\) is bounded, and both \(u\) and \(\nabla u\) vanish a.e.\ outside \(\{u\neq0\}\), so that \(\mathrm{Lip}_0(\Omega;X)\subset W^{1,p}_\mu(\Omega)\). 
We define the weighted Sobolev space $W^{1,p}_{\mu,0}(\Omega; X)$ as the closure, with respect to the norm~\eqref{eq:norm}, of (the restrictions to \(\Omega\) of the functions in) \(\mathrm{Lip}_0(\Omega;X)\). 
Being a closed subspace of \(W^{1,p}_{\mu}(\Omega)\), it inherits reflexivity. 
No condition is imposed on \(\partial\Omega\cap\partial X\).

\subsection*{Shape optimization problems}
The \emph{first Dirichlet eigenvalue} of the weight\-ed $p$-Laplacian on the open set $\Omega$ is defined variationally by
\begin{equation*}
\lambda_{1,p}^\mu(\Omega)
=
\inf
\left\{
\frac{\|\nabla u\|^p_{p,\mu,\Omega}}{\|u\|^p_{p,\mu,\Omega}} \,:\, u\in W^{1,p}_{\mu,0}(\Omega;X),\, u \neq 0
\right\}.
\end{equation*}

The admissible class is nonempty: given \(x_0\in\Omega\) and \(r>0\) with \(\overline{B_{r}(x_0)}\subset\Omega\), it is easy to see that the function \(u(x)=(r-\|x-x_0\|)_+\) belongs to \(\mathrm{Lip}_0(\Omega;X)\). 
Thus, \(\lambda^\mu_{1,p}(\Omega)<+\infty\).

For a continuous potential \(W\), the associated eigenvalue problem is understood in the weak sense: \(u\in W^{1,p}_{\mu,0}(\Omega;X)\setminus\{0\}\) is an eigenfunction with eigenvalue \(\lambda\) if
\begin{equation}
\label{eq:weak-eigen}
\int_\Omega|\nabla u|^{p-2}\langle\nabla u,\nabla\varphi\rangle\,\mathrm{d}\mu
=
\lambda\int_\Omega|u|^{p-2}u\,\varphi\,\mathrm{d}\mu,
\qquad
\forall\,\varphi\in W^{1,p}_{\mu,0}(\Omega;X).
\end{equation}
This is the Euler--Lagrange equation of the Rayleigh quotient above, so that its minimizers, when they exist, are eigenfunctions with eigenvalue \(\lambda^\mu_{1,p}(\Omega)\). 
The weak formulation~\eqref{eq:weak-eigen} encodes a Dirichlet boundary condition on \(\partial\Omega\cap X\), through the space \(W^{1,p}_{\mu,0}(\Omega;X)\), and a natural boundary condition on \(\partial\Omega\cap\partial X\). 
If, in addition, the potential \(W\) is of class \(C^1\), an integration by parts shows that~\eqref{eq:weak-eigen} is the weak form of the principal eigenvalue problem
\[
-\Delta_{p,\mu} u 
=
\lambda\,|u|^{p-2}u \quad\text{in }\Omega,
\]
where the weighted (or Witten-type) $p$-Laplacian of $u$ is given by the formal expression
\[
\begin{aligned}
-\Delta_{p,\mu}(u) 
&
=
-e^{-W}\div\bigl(e^{W}|\nabla u|^{p-2}\nabla u\bigr)
\\
&
=
-\div(|\nabla u|^{p-2}\nabla u) - |\nabla u|^{p-2}\langle \nabla W, \nabla u\rangle,
\end{aligned}
\]
that is, the usual $p$-Laplacian paired with a drift term, see, \emph{e.g.}, \cite{DMWX21}.

A classical problem in shape optimization is identifying sets $\Omega^*\subset X$ such that the \emph{Faber--Krahn inequality}
\begin{equation}
\label{eq:fk-ineq}
\lambda_{1,p}^\mu(\Omega^*) 
\le
\lambda_{1,p}^\mu(\Omega), 
\qquad
\forall\, 
\text{ open sets } 
\Omega\subset X 
\text{ with } 
\mu(\Omega)
=
\mu(\Omega^*),
\end{equation}
holds.

A similar shape optimization problem revolves around the \emph{$p$-torsional rigidity} of the weighted $p$-Laplacian on the open set $\Omega$, variationally defined as
\begin{equation*}
T_{p}^\mu(\Omega)
=
\sup
\left\{
\frac{\|u\|^p_{1,\mu,\Omega}}{\|\nabla u\|^p_{p,\mu,\Omega}} \,:\, u\in W^{1,p}_{\mu,0}(\Omega;X),\, \nabla u \neq 0
\right\}.
\end{equation*}
Two remarks on this definition are in order. %
First, different conventions exist: in place of the $p$-th power of the ratio, some take its $p'$-th power, where $p'$ is the conjugate exponent of $p$.
Second, usually one requires $u\neq 0$ rather than $\nabla u \neq 0$. %
For the Lebesgue measure on \(X=\mathbb{R}^N\), the two requirements are equivalent for functions in $W^{1,p}_{\mu,0}(\Omega;X)$. 
For a general measure, instead, nonzero constant functions may belong to $W^{1,p}_{\mu,0}(\Omega;X)$, and for these the ratio within the supremum would be ill-posed. 
This occurs, for instance, when $\Omega=X=\mathbb{R}^N$ and $\mu$ is the Gaussian measure.

The admissible class is nonempty, since, just as before, it contains \(u(x)=(r-\|x-x_0\|)_+\), so that \(T^\mu_p(\Omega)\in(0,+\infty]\).
The $p$-torsional rigidity is naturally associated with the torsion problem, whose weak formulation reads
\[
\int_\Omega|\nabla u|^{p-2}\langle\nabla u,\nabla\varphi\rangle\,\mathrm{d}\mu
=
\int_\Omega\varphi\,\mathrm{d}\mu,
\qquad
\forall\,\varphi\in W^{1,p}_{\mu,0}(\Omega;X),
\]
and which, if \(W\) is of class \(C^1\), is the weak form of the PDE
\[
-\Delta_{p,\mu} u 
=
1 \quad\text{in }\Omega,
\]
with the same boundary conditions as above. 
If it admits a solution \(u_\Omega\), the so-called \emph{torsion function}, then \(T^\mu_p(\Omega)=\bigl(\int_\Omega u_\Omega\,\mathrm{d}\mu\bigr)^{p-1}\), as follows by testing the equation with \(u_\Omega\) and from the fact that \(u_\Omega\) minimizes the functional \(v\mapsto\frac1p\int_\Omega|\nabla v|^p\,\mathrm{d}\mu-\int_\Omega v\,\mathrm{d}\mu\).

The $p$-torsional rigidity analog of the Faber--Krahn inequality~\eqref{eq:fk-ineq} is the  \emph{Saint-Venant inequality}, and one looks for sets $\Omega^*\subset X$, if they exist, such that
\begin{equation}
\label{eq:sv-ineq}
T_{p}^\mu(\Omega^*) 
\ge 
T_{p}^\mu(\Omega), 
\qquad
\forall
\text{ open sets }
\Omega\subset X 
\text{ with } 
\mu(\Omega)
=
\mu(\Omega^*),
\end{equation}
holds.
These two inequalities are particular instances of Faber--Krahn-type inequalities for the first $(p,q)$-eigenvalue of the weighted $p$-Laplacian, see \cref{rem:pq-eingevalue}.

\subsection*{State of the art}
We now give a short account, as complete as possible, of the settings in which these inequalities, or the relevant isoperimetric inequalities, have been established.

In the Euclidean setting, corresponding to $X=\mathbb{R}^N$ and $W\equiv c\in\mathbb{R}$, inequality~\eqref{eq:fk-ineq} was proved independently by Faber~\cite{Fab23} and Krahn~\cite{Kra25}, whereas inequality~\eqref{eq:sv-ineq} was proved by P\'olya~\cite{Pol48}.
In both cases, $\Omega^*=B$ is a Euclidean ball, and the proof relies on the \emph{Schwarz symmetrization}~\cite{HL30,Sch90book}, which entails the (Euclidean) \emph{P\'olya--Szeg\H{o} principle}~\cite{PS51book}.

In the Gaussian setting, corresponding to $X=\mathbb{R}^N$ and $W(x)=-\|x\|^2/2$, the Faber--Krahn inequality~\eqref{eq:fk-ineq} has been proved in~\cite{Ehr84} via a different symmetrization, known as \emph{Ehrhard symmetrization}~\cite{Ehr83}. 
This analysis has been extended in~\cite{BCF07} to a broader class of elliptic operators associated with the Gaussian measure, with refined estimates for the first eigenfunction. 
The same symmetrization allows one to prove the Saint-Venant inequality~\eqref{eq:sv-ineq} in this setting~\cite[Sect.~5]{Liv24} for \(p=2\); see also~\cite{CFMW26} for general $p$.

While the Euclidean and Gaussian settings were the first to attract attention, several other weighted settings have been addressed over the years, exploiting the same strategy. 
The first key step is proving the relevant isoperimetric inequality. %
In convex cones, for the Lebesgue measure, the relative isoperimetric inequality, with balls centered at the vertex as optimal sets, is due to Lions and Pacella~\cite{LP90}.
It has been extended by Cabr\'e, Ros-Oton, and Serra to homogeneous weights satisfying a suitable concavity condition~\cite{CROS12,CROS16}, see also~\cite{CR13} for monomial weights. 
Nonhomogeneous weights have been addressed by Brock, Chiacchio, and Mercaldo, in half-spaces~\cite{BCM12}, extending a two-dimensional result of Maderna and Salsa~\cite{MS81}, and in orthants~\cite{BCM14}. %
Further, the same authors derived necessary conditions on a weight for balls centered at the vertex of a cone to be isoperimetric~\cite{BCM12}. 
Balls centered at the origin are  also isoperimetric for smooth radial log-convex densities, as proved by Chambers~\cite{Cha19}. 
We also mention~\cite{DHHT12}, where it is shown that, in planar convex sectors endowed with the density \(\|x\|^{-2-k}\), \(k>0\), which blows up at the vertex, the isoperimetric sets are the complements of the balls centered at the vertex.

Transitioning to settings where the isoperimetric sets are half-spaces, beyond the usual Gaussian, we mention that Brock, Chiacchio, and Mercaldo also gave necessary conditions on a weight for coordinate half-spaces of a product domain to be isoperimetric, and proved that they are indeed isoperimetric for a broad class of Gauss-like product measures~\cite{BCM16}, see also~\cite{BCM08}.
Half-spaces also are the isoperimetric sets for anisotropic Gaussian measures~\cite{Yeh24} and for log-concave one-dimensional perturbations of the Gaussian measure~\cite{Ros14}. 

In several of these settings, the isoperimetric inequality has then been combined with suitable rearrangements to derive P\'olya--Szeg\H{o} principles, comparison results for degenerate elliptic equations and, in some cases, Faber--Krahn or Saint-Venant inequalities, mostly for \(p=2\), see~\cite{BCM08,BCM12,BCM16,CM25,MS81,NV25}.

\subsection*{Abstract Talenti symmetrization}
As the above overview shows, a common tool behind these results is a symmetrization tailored to the isoperimetric sets of the measure, which entails a P\'olya--Szeg\H{o} principle: a competitor \(u\) on \(\Omega\) is transformed into a competitor \(u^\sharp\) on a set \(\Omega^\sharp\), decreasing the \(L^p\)-norm of the gradient while keeping unchanged the \(L^p\)-norm of the function. 
Roughly speaking, superlevel sets of \(u^\sharp\) at a given height are isoperimetric sets having the same measure as the corresponding superlevel set of \(u\). 

In this note we revisit the Faber--Krahn and Saint-Venant inequalities in the weighted setting of measures of the form \eqref{eq:measures} by exploiting an abstract symmetrization, first introduced by Talenti in~\cite{Tal97}, which, again, replaces each superlevel set of \(u\) with the set having the same measure in a fixed nested family of isoperimetric sets.
Namely, it is driven by an underlying smooth function \(f:X\to[0,1]\) without critical points in \(\{0<f<1\}\), see \cref{def:talenti}: the Talenti symmetrization \(u^\sharp\) of \(u\) is the function whose superlevel sets are the superlevel sets of \(f\) having the same \(\mu\)-measure as the corresponding superlevel sets of \(|u|\), see~\eqref{eq:def-symmetrized}. 
Talenti's weighted P\'olya--Szeg\H{o} principle~\cite[Main Thm.]{Tal97} states that
\[
\int_X|\nabla u^\sharp|^p\,\mathrm{d}\mu\le\int_X|\nabla u|^p\,\mathrm{d}\mu
\]
for Lipschitz functions \(u\) on \(X\) with \(\mu(\{u\neq0\})<+\infty\), provided the level sets of \(f\) are connected, the superlevel sets of \(f\) saturate a suitable weighted relative isoperimetric inequality, and the lines of steepest ascent of \(f\) are straight; see \cref{sec:assumptions}, \cref{thm:talenti}, and the subsequent discussion.
The advantage of Talenti's approach is that, once such an $f$ is available, the P\'olya--Szeg\H{o} principle, the Faber--Krahn and Saint-Venant inequalities follow in a unified way, without having to construct an ad hoc rearrangement for each measure. 

\subsection*{Description of results}
Our first main result, \cref{thm:polya-szego}, shows that Talenti's principle, originally formulated for Lipschitz functions on the whole of $X$, extends to the space \(W^{1,p}_{\mu,0}(\Omega;X)\) of Sobolev functions vanishing on \(\partial\Omega\cap X\), for open subsets $\Omega\subset X$ of finite measure, provided the ambient space \(X\) is convex. 
The extension is obtained via a density argument in weighted Sobolev spaces, and immediately yields as corollaries the corresponding Faber--Krahn and Saint-Venant inequalities, see \cref{cor:fk-ineq,cor:sv-ineq}, as well as Faber--Krahn-type inequalities for the first \((p,q)\)-eigenvalue, see \cref{rem:pq-eingevalue}.
In \cref{rem:threshold}, we point out that it is enough to know that the superlevel sets of \(f\) are isoperimetric for volumes below a given threshold, at the price of restricting the Faber--Krahn and Saint-Venant inequalities to sets whose measure lies below the same threshold. 

Our second result, \cref{thm:geo-uc}, concerns the structure of the functions \(f\) satisfying the assumptions of Talenti's theorem; it does not involve the isoperimetric properties of \(\mu\), but only the geometry of \(f\).

In \cref{thm:geo-uc}, we prove that \(|\nabla f|=g(f)\) for a continuous profile \(g\), that is, \(|\nabla f|\) is constant on each level set of \(f\), and that the level sets of \(f\) are quantitatively separated: if \(G\) is a primitive of \(1/g\), then
\[
\dist\bigl(\{f\le t_1\};\{f\ge t_2\}\bigr)\ge G(t_2)-G(t_1)>0,
\qquad
0<t_1<t_2<1.
\]
This separation is the key ingredient in the proof of \cref{thm:polya-szego}: it guarantees that the symmetrization of a Lipschitz function supported away from \(\partial\Omega\cap X\) is supported away from \(\partial\Omega^\sharp\cap X\), so that it can be approximated within \(W^{1,p}_{\mu,0}(\Omega^\sharp;X)\).

The assumptions of Talenti's theorem are actually much more rigid than they may appear. In the forthcoming paper~\cite{BS-geom}, we show that the level sets of \(f\) are parallel hypersurfaces, \(f\) being a transnormal function, and that the isoperimetric property of its superlevel sets forces them to be either parallel half-spaces, or concentric balls, or complements of tubes around an affine subspace of codimension at least two, the last two alternatives being possible only if \(\mu(X)=+\infty\).

Summing up, our framework covers all the settings recalled above at once, whether the isoperimetric sets are balls or half-spaces, by a single argument whose only input is the relative isoperimetric inequality, and it yields the Faber--Krahn, Saint-Venant, and \((p,q)\)-eigenvalue inequalities for every \(p\in(1,+\infty)\); several of them, in particular for \(p\neq2\), appear not to have been stated before. %
A precise account of the known results and of the novelties is given in \cref{sec:examples}, setting by setting, where we also discuss convex polytopes, in which the isoperimetric sets are known only for small volumes, see \cref{sec:small-volumes}, and planar convex sectors with a density blowing up at the vertex, in which the isoperimetric sets are complements of balls, see \cref{sec:complements}.

\subsection*{Organization of the paper}
In \cref{sec:assumptions}, we fix the setting and collect all the assumptions on \(\mu\) and \(f\) used in the paper.
In \cref{sec:symmetrization}, we recall Talenti symmetrization in the present setting, introduce the associated symmetrized sets, prove some basic properties, and recall Talenti's weighted P\'olya--Szeg\H{o} inequality for Lipschitz functions.
In \cref{sec:structure}, we prove the separation of the level sets of the admissible functions \(f\).
In \cref{sec:polya-szego}, we extend the P\'olya--Szeg\H{o} principle from Lipschitz functions to \(W^{1,p}_{\mu,0}(\Omega;X)\), thereby deriving the corresponding Faber--Krahn inequality~\eqref{eq:fk-ineq} and Saint-Venant inequality~\eqref{eq:sv-ineq}, and we discuss the case of isoperimetric sets known only for small volumes.
In \cref{sec:examples}, we exhibit several classes of measures for which the hypotheses of Talenti's theorem are satisfied, at least for small volumes.
Finally, in \cref{sec:future}, we describe two further developments, namely the Kohler-Jobin inequality and quantitative versions of the Saint-Venant and Faber--Krahn inequalities, which will appear in forthcoming works.

\subsection*{Acknowledgments} 

The present note stems from G.B.'s Master's Thesis. 
It has been written up while G.S.\ visited G.B.\ at the Goethe-Universität Frankfurt, whose hospitality is kindly acknowledged. 
Both authors are members of INdAM--GNAMPA. G.B.\ is supported by the Deutsche Forschungsgemeinschaft (DFG, German Research Foundation), Project No.~555837013.

\section{Setting and assumptions}
\label{sec:assumptions}

Throughout the paper, \(X\subset\R^N\) is an open and connected set, and \(\mu=e^W\mathcal{L}^N\) is as in~\eqref{eq:measures}, with \(W\) continuous on \(X\); the convexity of \(X\) is explicitly assumed whenever it is needed, namely from \cref{thm:geo-uc} on.
The \(\mu\)-perimeter of a Borel set \(\Omega\subset X\) relative to \(X\) is
\begin{equation}
\label{eq:P-mu}
\Per_\mu(\Omega; X) = \int_{\partial^* \Omega \cap X} e^{W(x)}\,\mathrm{d}\mathcal{H}^{N-1}(x),
\end{equation}
where \(\partial^*\Omega\) is the reduced boundary of \(\Omega\), refer to~\cite{Mag12book}; in particular, \(\partial X\) is not counted.

Throughout the paper, \(f:X\to[0,1]\) is a smooth function, and, for \(t\in[0,1]\), we let \(\Phi_f(t)=\mu(\{\,x\in X\,:\,f(x)>t\,\})\), see also~\eqref{eq:distribution-function} below.
We collect here all the assumptions on \(f\) and \(\mu\) that we shall use; each statement below recalls the ones it needs.

The first three are normalization and nondegeneracy conditions, under which the Talenti symmetrization is defined, see \cref{def:talenti}:
\begin{enumerate}[label=\textup{(N\arabic*)}, ref=\textup{(N\arabic*)}, leftmargin=*, widest=3]
\item
\label{hp:N1}
\(\inf f=0\) and \(\sup f=1\), with \(\mu\bigl(f^{-1}(0)\bigr)=\mu\bigl(f^{-1}(1)\bigr)=0\);
\item
\label{hp:N2}
\(\nabla f(x)\neq0\) for all \(x\in X\) with \(0<f(x)<1\);
\item
\label{hp:N3}
\(\Phi_f(t)<+\infty\) for all \(t\in(0,1)\), or \(\mu(X)<+\infty\).
\end{enumerate}
We let \(X_{\mathrm{reg}}=\{\,x\in X\,:\,0<f(x)<1\,\}\) and, for \(t\in(0,1)\), we denote by \(M_t=\{\,x\in X\,:\,f(x)=t\,\}\) the level sets of \(f\), which, under~\ref{hp:N2}, are smooth hypersurfaces.
Further, the unit normal
\begin{equation}
\label{eq:def-N}
\mathbf{N}(x)
=
\frac{\nabla f(x)}{\|\nabla f(x)\|},
\qquad x\in X_{\mathrm{reg}},
\end{equation}
is smooth on \(X_{\mathrm{reg}}\), and the \emph{lines of steepest ascent} of \(f\) are the maximal integral curves of \(\mathbf{N}\) in \(X_{\mathrm{reg}}\), that is, the maximal solutions of \(\dot{x}(s)=\mathbf{N}(x(s))\), where \(s\) is the arclength parameter.

The next four are the assumptions of Talenti's theorem, see \cref{thm:talenti}:
\begin{enumerate}[label=\textup{(T\arabic*)}, ref=\textup{(T\arabic*)}, leftmargin=*, widest=4]
\item
\label{hp:T1}
\emph{isoperimetric inequality:} there exists a continuous function \(I\ge0\) on \([0,\mu(X))\), positive on \((0,\mu(X))\), such that
\[
I(\mu(E))\le\Per_\mu(E;X),
\quad \forall
\text{ Borel sets \(E\subset X\) with \(\mu(E)<\mu(X)\);}
\]
\item
\label{hp:T2}
\emph{saturation:} \(\Per_\mu(\{f>t\};X)=I(\Phi_f(t))\) for all \(t\in(0,1)\);
\item
\label{hp:T3}
\emph{connectedness:} the level sets \(M_t\), \(t\in(0,1)\), are connected;
\item
\label{hp:T4}
\emph{straightness:} the lines of steepest ascent of \(f\) are segments, that is, \(\mathrm{d}^2x/\mathrm{d}s^2=0\) along each of them.
\end{enumerate}

When the isoperimetric sets are only known for small volumes, see \cref{rem:threshold},~\ref{hp:T1}--\ref{hp:T3} are replaced by the following ones, where \(\overline{m}\in(0,\mu(X))\) is a given threshold:
\begin{enumerate}[label=\textup{(T\arabic*$'$)}, ref=\textup{(T\arabic*$'$)}, leftmargin=*, widest=4]
\item
\label{hp:T1'}
there exists a continuous function \(I:[0,\overline{m}]\to[0,+\infty)\), positive on \((0,\overline{m}]\), such that \(I(\mu(E))\le\Per_\mu(E;X)\) for all Borel sets \(E\subset X\) with \(\mu(E)\le\overline{m}\);
\item
\label{hp:T2'}
\(\Per_\mu(\{f>t\};X)=I(\Phi_f(t))\) for all \(t\in(0,1)\) with \(\Phi_f(t)\le\overline{m}\);
\item
\label{hp:T3'}
the level sets \(M_t\) are connected for all \(t\in(0,1)\) with \(\Phi_f(t)<\overline{m}\).
\end{enumerate}

\section{Talenti symmetrization}
\label{sec:symmetrization}

In this section we recall the symmetrization defined in~\cite{Tal97}.
The continuity of \(W\) can be considerably weakened regarding the results of this section preceding \cref{sec:tal-polya-szego}, as the next remark explains.

\begin{remark}
The constructions and the results of this section, up to \cref{prop:cnt-sym}, only use two properties of \(\mu\): \(\mu\)-negligible and \(\mathcal{L}^N\)-negligible sets coincide, and nonempty open subsets of \(X\) have positive \(\mu\)-measure. 
Thus, they hold, with the same proofs, for every \(\sigma\)-finite measure \(\mu\) on \(X\) mutually absolutely continuous with \(\mathcal{L}^N\). 
Any such measure can be written as \(\mu=e^W\mathcal{L}^N\), where \(W=\ln w\) and \(w>0\) is its Radon--Nikodym derivative, but then \(W\) is merely measurable and defined \(\mathcal{L}^N\)-a.e.; in particular, no integrability of negative powers of the weight, such as~\eqref{eq:kufner-opic}, can be expected.
The continuity of \(W\) is used in a critical way from \cref{sec:tal-polya-szego} on: it gives a meaning to the weighted perimeter~\eqref{eq:P-mu}, which integrates \(e^W\) over an \(\mathcal{L}^N\)-negligible set, and it ensures, via~\eqref{eq:kufner-opic}, that the weighted Sobolev spaces of the introduction have the usual properties.
\end{remark}

\begin{definition}[$\mu$-distribution function]
Let $v:X \to \R$ be a real-valued $\mu$-measurable function.
Assume either that the superlevel sets $\{\,|v(x)|>t\,\}$ have finite $\mu$-measure for all $t>0$ or that \(\mu(X)<+\infty\). 
The \emph{$\mu$-distribution function} of $v$, denoted by $\Phi_v^\mu$, is the map that associates to any $t>0$ the $\mu$-measure of the superlevel set $\{\,|v|>t\,\}$, \emph{i.e.}, the map from $[0,+\infty)$ into $[0,+\infty]$ defined by 
\begin{equation}
\label{eq:distribution-function}
\Phi^\mu_v(t)=\mu\bigl(\{\,x \in X:|v(x)|>t\,\}\bigr).
\end{equation}
\end{definition}

If no confusion can arise, we will drop the superscript $\mu$ and denote it simply by $\Phi_v$ and refer to it simply as the distribution function.
It is clear that $\Phi_v$: decreases monotonically; is right continuous\footnote{If, for all $t>0$, $\mu(\{\, x\in X: |v(x)|=t\,\})=0$, it improves to continuity.}.
Moreover, one has
\[
\Phi_v(0)=\mu\bigl(\{\,x \in X:v(x)\neq0\,\}\bigr)
\]
and 
\[
\{\,t\geq0:\Phi_v(t)=0\,\}=[\|v\|_{\infty},+\infty),
\]
where the essential supremum is with respect to the measure $\mu$.

\begin{definition}[$\mu$-decreasing rearrangement]
\label{def:d-rearrangement}
Let $v:X \to \R$ be a real-valued $\mu$-measurable function. 
Assume either that the superlevel sets $\{\,|v(x)|>t\,\}$ have finite $\mu$-measure for all $t>0$ or that \(\mu(X)<+\infty\).
The \emph{$\mu$-decreasing rearrangement} of $v$ is defined by
\begin{equation}
\label{eq:decreasing-rearrangement}
v^{\mu,*}(s)=\inf\{\,t \geq 0 : \Phi_v^\mu(t)\leq s\,\},
\end{equation}
that is, \(v^{\mu,*}\) is the (right-continuous) generalized inverse of the nonincreasing function \(\Phi_v^\mu\).
\end{definition}

As for the distribution function, if no confusion can arise, we will drop the superscript $\mu$ and denote the $\mu$-decreasing rearrangement of $v$ with $v^*$ and refer to it simply as the decreasing rearrangement.

The decreasing rearrangement $v^*$ can also be interpreted as the distribution function of $\Phi_v$ with respect to the Lebesgue measure.
Hence, the function $v^*$ is a decreasing, right-continuous map from $[0,+\infty)$ into $[0,+\infty]$, satisfying $v^*(0)=\|v\|_\infty$ and
\[
\{\, s\geq 0 : v^*(s)=0\,\}
=
\bigl[\,\mu\bigl(\{\,x \in X: v(x)\neq0\,\}\bigr),+\infty\bigr).
\]
Moreover, for every $y>0$ we have
\begin{equation}
\label{eq:superlevelsets-v*}
\{\, s\geq 0 : v^*(s)>y\,\}=[\,0,\Phi_v(y)).
\end{equation}
\begin{definition}[$(\mu,f)$-Talenti symmetrization] \label{def:talenti}
Assume~\ref{hp:N1}--\ref{hp:N3}.
Given a nonempty, proper open subset $\Omega\subsetneq X$ and a $\mu$-measurable function $u:\Omega \to \R$ such that either the superlevel sets $\{\,|u(x)|>t\,\}$ have finite $\mu$-measure for all $t>0$ or \(\mu(X)<+\infty\), the \emph{$(\mu,f)$-Talenti symmetrization} of $u$ is the function $u^{\mu,\sharp}:X\to[0,+\infty]$ given by
\begin{equation}
\label{eq:def-symmetrized}
u^{\mu,\sharp}(x)=\tilde{u}^{\mu,*}(\Phi^\mu_f(f(x))),
\end{equation}
where $\tilde{u}$ is the extension to zero of $u$ to $X\setminus\Omega$.
\end{definition}

Notice that \(u^\sharp\) may take the value \(+\infty\) only on the \(\mu\)-negligible set \(\{f=1\}\) (see \cref{lem:Phi-f} below), where \(u^\sharp=\tilde u^*(0)=\|u\|_{\infty}\).
If no confusion can arise, we will drop the superscript $\mu$ and denote it simply by $u^\sharp$, and also drop the tilde on \(\tilde u\) and simply write \(u\). 
For the sake of ease, we will say Talenti symmetrization, without making explicit the dependence on the measure $\mu$ and on the function $f$.
We start collecting the basic properties of the distribution function of \(f\).

\begin{lemma}
\label{lem:Phi-f}
Assume~\ref{hp:N1}--\ref{hp:N3}. Then, \(\Phi_f\) is continuous and strictly decreasing on \([0,1]\), with \(\Phi_f(0)=\mu(X)\) and \(\Phi_f(1)=0\). 
\end{lemma}

In other words, the above lemma says that \(\Phi_f\) is a homeomorphism of \([0,1]\) onto \([0,\mu(X)]\) (onto \([0,+\infty]\), if \(\mu(X)=+\infty\)), and its inverse is the restriction of \(f^*\) to \([0,\mu(X)]\).

\begin{proof}
First, by~\ref{hp:N1}, \(\Phi_f(0)=\mu(\{f>0\})=\mu(X)\) and \(\Phi_f(1)=\mu(\emptyset)=0\).

Second, the function \(\Phi_f\) is right-continuous and, owing to~\ref{hp:N3}, its jump at \(t\in(0,1]\) equals \(\mu(\{f=t\})\). 
This quantity vanishes: for \(t=1\) by~\ref{hp:N1}; for \(t\in(0,1)\) because, by~\ref{hp:N2}, \(\{f=t\}\) is a smooth hypersurface, hence it is \(\mathcal{L}^N\)-negligible, thus \(\mu\)-negligible. 

Third, let \(0\le t_1<t_2\le 1\). Since \(f\) is continuous on the connected set \(X\) with \(\inf f=0\) and \(\sup f=1\), one has \(f(X)\supset(0,1)\), so that \(\{t_1<f<t_2\}\) is a nonempty open set. %
Since \(\mu\) is strictly positive,
\[
\Phi_f(t_1)-\Phi_f(t_2)
=
\mu(\{t_1<f\le t_2\})
\ge\mu(\{t_1<f<t_2\})>0,
\]
with the obvious meaning if \(\Phi_f(t_1)=+\infty\). 

Finally, for a continuous, strictly decreasing function, the generalized inverse~\eqref{eq:decreasing-rearrangement} coincides with the inverse.
\end{proof}

\begin{definition}[Symmetrization of a set]
Assume~\ref{hp:N1}--\ref{hp:N3}, and let \(\Omega\subsetneq X\) be a nonempty open set with \(\mu(\Omega)<+\infty\). We define the Talenti-symmetrized set of \(\Omega\) as
\(
\Omega^\sharp = \{x\in X\,:\, \chi_\Omega^\sharp(x)>0 \}.
\)
\end{definition}

We now prove the following lemma, which characterizes the symmetrized set $\Omega^\sharp$ as the union of the positivity sets of the Talenti symmetrizations of all $\mu$-measurable functions defined on $\Omega$. 
In view of it, given a $\mu$-measurable $u$ defined on such an open set $\Omega$, we can treat its Talenti symmetrization $u^\sharp$ as defined on $\Omega^\sharp$.

\begin{lemma}
\label{lem:levelset-symmetrizedset-inclusion}
Assume~\ref{hp:N1}--\ref{hp:N3}, and let \(\Omega\subsetneq X\) be a nonempty open set with \(\mu(\Omega)<+\infty\).
Then,
\[
\bigcup\, \{ x\,:\, u^\sharp (x) > 0\} 
=
\Omega^\sharp,
\]
where the union is taken among all $\mu$-measurable functions $u$ defined on $\Omega$.
\end{lemma}

\begin{proof}
Let us show the two opposite inclusions. 
On the one hand, since $\chi_\Omega$ is $\mu$-measurable, one inclusion is trivial, by definition of $\Omega^\sharp$.

On the other hand, let us start writing down the positivity set of $\chi_\Omega^\sharp$. By definition~\eqref{eq:def-symmetrized}, setting $\sigma=\sigma(x)=\Phi_f(f(x))$, we have
\[
\chi_\Omega^\sharp(x) 
=
\chi_\Omega^* (\sigma(x)).
\]
Recall that by definition~\eqref{eq:decreasing-rearrangement}
\[
\chi_\Omega^*(s) 
=
\inf\{t\ge 0\,:\, \Phi_{\chi_\Omega}(t)\le s\},
\]
where the distribution function of $\chi_\Omega$, refer to~\eqref{eq:distribution-function}, is
\[
\Phi_{\chi_\Omega}(t) 
=
\mu\bigl(\{z\in X\,:\, \chi_\Omega(z) > t \}\bigr).
\]
Thus, we have
\[
\Phi_{\chi_\Omega}(t) 
=
\left\{
\begin{aligned}
&\mu(\Omega), \quad &&\text{if $t\in[0,1)$,}\\
&0, &&\text{otherwise.}
\end{aligned}
\right.
\]
Therefore, we get
\begin{equation}
\label{eq:chi_char}
\chi_\Omega^*(s) 
=
\left\{
\begin{aligned}
&0, \quad &&\text{if $s\ge\mu(\Omega)$,}\\
&1, &&\text{otherwise.}
\end{aligned}
\right.
\end{equation}

Let now $u$ be any $\mu$-measurable function defined on $\Omega$, let $x \in \{u^\sharp>0\}$, and $\sigma(x)=\Phi_f(f(x))$ as before. 
By definition of $u^\sharp$ and the characterization of $u^*$ given in~\eqref{eq:superlevelsets-v*},
\(
\sigma(x) \in [0, \Phi_u(0)).
\)
Therefore, we have $\sigma(x) < \Phi_u(0) \le \mu(\Omega)$. 
Hence, by~\eqref{eq:chi_char}, $x\in\Omega^\sharp$.
\end{proof}

Further, the symmetrized set \(\Omega^\sharp\) is a superlevel set of \(f\) with the same measure as \(\Omega\) and, whenever \(\mu(\Omega)\neq \mu(X)\), its boundary relative to \(X\) coincides with the corresponding level set of \(f\).

\begin{lemma}
\label{lem:sym-set}
Assume~\ref{hp:N1}--\ref{hp:N3}, let \(\Omega\subsetneq X\) be a nonempty open set with \(\mu(\Omega)<+\infty\), and let \(t_\Omega=\Phi_f^{-1}(\mu(\Omega))\). 
Then, \(t_\Omega\in[0,1)\),
\[
\Omega^\sharp=\{f>t_\Omega\},
\qquad
X\setminus\Omega^\sharp=\{f\le t_\Omega\},
\qquad
\mu(\Omega^\sharp)=\mu(\Omega).
\]
In particular, \(\Omega^\sharp\) is open and it depends on \(\Omega\) only through \(\mu(\Omega)\). 
If, moreover, \(\mu(\Omega)<\mu(X)\), then \(t_\Omega\in(0,1)\) and \(\partial\Omega^\sharp\cap X=\{f=t_\Omega\}\).
\end{lemma}

\begin{proof}
Since \(\Omega\) is a nonempty open set and \(\mu\) is strictly positive, we have \(0<\mu(\Omega)\le\mu(X)\).
Thus, by \cref{lem:Phi-f}, \(t_\Omega\) is well-defined and \(t_\Omega<1\). %
By~\eqref{eq:chi_char}, recalling that in the proof of \cref{lem:levelset-symmetrizedset-inclusion} we set \(\sigma(x)=\Phi_f(f(x))\), we have \(\chi_\Omega^\sharp(x)>0\) if and only if \(\Phi_f(f(x))<\mu(\Omega)=\Phi_f(t_\Omega)\). 
By \cref{lem:Phi-f}, \(\Phi_f\) is strictly decreasing, thus this is equivalent to \(f(x)>t_\Omega\). 
This proves the first two equalities, and gives \(\mu(\Omega^\sharp)=\Phi_f(t_\Omega)=\mu(\Omega)\). %

If moreover \(\mu(\Omega)<\mu(X)=\Phi_f(0)\), then \(t_\Omega>0\). 
By continuity of \(f\), one has the set inclusion \(\partial\Omega^\sharp\cap X\subset\{f=t_\Omega\}\).
Conversely, let \(x\in X\) with \(f(x)=t_\Omega\in(0,1)\), so that by the first part of the proof \(x\notin\Omega^\sharp\). 
Recall that by assumption~\ref{hp:N2}, we have \(\nabla f(x)\neq0\). 
Since \(X\) is open, the function
\[
\phi(\varepsilon)
=
f(x+\varepsilon\nabla f(x))
\]
is well-defined for \(\varepsilon\) small enough. 
Moreover, it satisfies
\[
\phi(0)=t_\Omega
\qquad
\text{and}
\qquad
\phi'(0)=|\nabla f(x)|^2>0.
\]
Hence, for all \(0<\varepsilon\ll 1\), we have \(\phi(\varepsilon)>t_\Omega\), that is, \(x+\varepsilon\nabla f(x)\in\Omega^\sharp\). 
Letting \(\varepsilon\to0^+\), we get \(x\in\overline{\Omega^\sharp}^X\setminus\Omega^\sharp=\partial\Omega^\sharp\cap X\).
\end{proof}

\begin{proposition}[Continuity of the symmetrization]
\label{prop:cnt-sym}
Assume~\ref{hp:N1}--\ref{hp:N3}, and let \(\Omega\subsetneq X\) be a nonempty open set with \(\mu(\Omega)<+\infty\). 
For every $u \in L^p_\mu(\Omega)$, its Talenti symmetrization $u^\sharp$ belongs to $L^p_\mu(\Omega^\sharp)$ and
\begin{equation}
\label{eq:Lp-norm-preserved}
\|u\|_{p,\mu,\Omega} = \|u^\sharp\|_{p,\mu,\Omega^\sharp}.
\end{equation}
Moreover, the Talenti symmetrization map \(T : u\mapsto u^\sharp\) is an \(L^p\)-nonexpansive map on the cone of nonnegative functions in \(L^p_\mu(\Omega)\), \emph{i.e.},
\begin{equation}
\label{eq:lp-non-expansive}
\|u^\sharp - v^\sharp\|_{p,\mu,\Omega^\sharp}
\le
\|u-v\|_{p,\mu,\Omega},
\qquad
\forall u, v\ge 0,\quad \text{\(\mu\)-a.e.}
\end{equation}
Therefore, it is continuous on such cone.
\end{proposition}

\begin{proof}
The construction of the Talenti symmetrization ensures that every superlevel set of $u^\sharp$ coincides with a superlevel set of $f$.  
Indeed, for every $t>0$,
\[
\{\,x\in X:u^\sharp(x)>t\,\}
=
\{\,x\in X:f(x)>f^*(\Phi_{{u}}(t))\,\}, 
\]
also owing to the fact that, by \cref{lem:Phi-f}, \(\Phi_f\) is a homeomorphism of \([0,1]\) onto \([0,\mu(X)]\), whose inverse is the restriction of \(f^*\) to \([0,\mu(X)]\). 
Moreover, $\Phi_{u^\sharp}(t)=\Phi_{{u}}(t)$.  %
Then, by the layer-cake formula (see~\cite[Ex.~1.13]{Mag12book}), the definition~\eqref{eq:distribution-function} of $\mu$-distribution of ${u}$, the previous equality, the layer-cake formula once again, and \cref{lem:levelset-symmetrizedset-inclusion}, we have
\begin{align*}
    \|u\|^p_{p,\mu,\Omega}
    =
    \|{u}\|^p_{p,\mu,X}
    &= p\int_0^{+\infty} t^{p-1}\,\mu\bigl(\{\,x\in X:|{u}(x)|>t\,\}\bigr)\,\di t \\
    &= p\int_0^{+\infty} t^{p-1}\Phi_{{u}}(t)\,\di t
     = p\int_0^{+\infty} t^{p-1}\Phi_{u^\sharp}(t)\,\di t \\
    &= \|u^\sharp\|^p_{p,\mu,X}
    = \|u^\sharp\|^p_{p,\mu,\Omega^\sharp}.
\end{align*}

The nonexpansivity of the Talenti symmetrization from \(L^p_\mu(\Omega)\) to \(L^p_\mu(\Omega^\sharp)\) is an immediate consequence of~\cite{CT80}.
The cited results are stated for maps from an \(L^p_\mu(\Omega)\) space into itself, with \(\mu\) a nonnegative measure, but proofs can be repeated verbatim for maps \(L^p_\mu(\Omega_\mu) \to L^p_\nu(\Omega_\nu)\), where the spaces are taken with respect to different nonnegative measures, up to additionally requiring \(\mu(\Omega_\mu) =\nu(\Omega_\nu)\), which in our case holds by \cref{lem:sym-set}.
We also refer to~\cite[Ex.~2]{CT80} where rearrangements are explicitly considered.

Since the Talenti symmetrization
\(
T : u \mapsto u^\sharp
\)
is integral preserving and order preserving for nonnegative functions, by~\cite[Prop.~1(a) and~(c)]{CT80}, it is an \(L^1\)-nonexpansive map between the cone of nonnegative functions in \(L^1_\mu(\Omega)\) and \(L^1_\mu(\Omega^\sharp)\).

Further, the symmetrization commutes with the addition of nonnegative constants on the cone of nonnegative functions: if \(u\ge0\) and \(c\ge0\), where \(u+c\) is understood on \(\Omega\) and extended by zero outside, then \((u+c)^*=u^*+c\) on \([0,\mu(\Omega))\), so that
\[
T(u+c)=T(u)+c\qquad\text{on }\Omega^\sharp.
\]
Hence, given \(u,v\ge0\) and \(k=\esssup_{\Omega}|u-v|\), from \(u\le v+k\) and the order preservation we get \(T(u)\le T(v+k)=T(v)+k\) on \(\Omega^\sharp\). 
Exchanging the roles of \(u\) and \(v\), the symmetrization is an \(L^\infty\)-nonexpansive map between the cone of nonnegative functions in \(L^\infty_\mu(\Omega)\) and \(L^\infty_\mu(\Omega^\sharp)\), see also~\cite[Prop.~2]{CT80}.

By standard nonlinear interpolation, it is an \(L^p\)-nonexpansive map between the cone of nonnegative functions in \(L^p_\mu(\Omega)\) and \(L^p_\mu(\Omega^\sharp)\), refer to~\cite[Ex.~2, Eqs.~(9)--(10)]{CT80} or to~\cite{Tar72}, so that~\eqref{eq:lp-non-expansive} follows.
\end{proof}

\subsection{Talenti's \texorpdfstring{P\'olya--Szeg\H{o}}{Polya--Szego} principle}
\label{sec:tal-polya-szego}

P\'olya--Szeg\H{o} principles arising from symmetrizations ensure that integrals of Young functions $F(\,\cdot\,)$, depending only on $|\nabla u|$, decrease when evaluated on $|\nabla u^\sharp|$, and we refer the interested reader to~\cite{PS51book} for its first appearance, in the Euclidean context. 
The usual symmetrizations rely on transforming level sets into isoperimetric sets. 
Talenti showed in~\cite{Tal97} that this principle stays true for Lipschitz functions also for his symmetrization provided that: the measure space  $(X,\mu)$ possesses a relative isoperimetric inequality with respect to the notions of weighted volume and perimeter, see~\ref{hp:T1}; $f$ is chosen so that its superlevel sets are isoperimetric sets and its level sets are connected, see~\ref{hp:T2} and~\ref{hp:T3}, and it satisfies a suitable steepest ascent condition, see~\ref{hp:T4}. 

We recall below the theorem, tailored to the case of our interest, $F(z) = |z|^p$. 

\begin{theorem}[P\'olya--Szeg\H{o} principle of~\cite{Tal97}]
\label{thm:talenti}
Assume~\ref{hp:N1}--\ref{hp:N3} and~\ref{hp:T1}--\ref{hp:T4}.
Then, given any bounded Lipschitz function $u:X\to \R$ such that \(\mu(\{u\neq0\})<+\infty\), the Talenti symmetrization $u^\sharp$ of $u$ (defined by~\eqref{eq:def-symmetrized} with \(\tilde u=u\)) is locally Lipschitz and, for every $p\in(1,+\infty)$, it holds
\begin{equation}
\label{eq:PS-Lip}
\int_{X} |\nabla u|^p \, \di \mu 
\geq 
\int_{X} |\nabla u^\sharp|^p \, \di \mu.
\end{equation}
\end{theorem}

We remark that Talenti does not explicitly mention that the symmetrized function is locally Lipschitz, but this can be derived by closely inspecting his proof. 
A global Lipschitz bound, which is what is needed in the sequel, is a consequence of~\eqref{eq:PS-Lip}, see \cref{lem:lip-bound} below.

Further, he does not use as perimeter the weighted $(N-1)$-dimensional Hausdorff measure~\eqref{eq:P-mu} but rather the relative dual-sided Minkowski content~\cite[Eq.~(1.7)]{Tal97}. 
Nevertheless, in the proof he uses the weighted variational relative perimeter to exploit coarea formula. 
This is well-known to be equivalent to~\eqref{eq:P-mu}, which coincides with the relative dual-sided Minkowski content for almost every superlevel set of a Lipschitz function.

Moreover, in~\cite[Main Thm.]{Tal97} the function \(I\) (denoted there by \(q\)) is required to be smooth on \([0,\mu(X))\). 
However, only its continuity and positivity on \((0,\mu(X))\) are used in the proof, see~\cite[Eqs.~(3.4) and~(3.10)]{Tal97}, while smoothness at \(0\) fails already for the Euclidean isoperimetric profile \(I(s)=c_N s^{1-1/N}\), considered in~\cite[Ex.~1]{Tal97}.
We have also made explicit the requirement~\ref{hp:T3} that the level sets of \(f\) are connected. 
Talenti's argument uses the fact that \(|\nabla f|\) is constant on each level set of \(f\), where, in its proof, this is derived from~\ref{hp:T4} by means of~\cite{Mag91}. 
As shown in the proof of \cref{thm:geo-uc}\ref{item:g} below, condition~\ref{hp:T4} only yields that \(|\nabla f|\) is constant on each \emph{connected component} of the level sets of \(f\), whence the need of the connectedness assumption.

As a consequence of~\eqref{eq:PS-Lip} holding for every \(p\), the Talenti symmetrization does not increase the Lipschitz constant, provided \(X\) is convex.

\begin{lemma}[Global Lipschitz bound]
\label{lem:lip-bound}
Assume~\ref{hp:N1}--\ref{hp:N3} and~\ref{hp:T1}--\ref{hp:T4}, and let \(u:X\to\R\) be a bounded Lipschitz function with \(\mu(\{u\neq0\})<+\infty\). 
Then,
\[
\|\nabla u^\sharp\|_{L^\infty(X)}
\le
\|\nabla u\|_{L^\infty(X)}.
\]
In particular, if \(X\) is convex, then \(u^\sharp\) is Lipschitz on \(X\), and its Lipschitz constant does not exceed that of \(u\).
\end{lemma}

\begin{proof}
Let \(m=\mu(\{u\neq0\})\). 
If \(m=0\), then \(u^\sharp\equiv0\) and there is nothing to prove. 
Let then \(m\neq 0\). By~\eqref{eq:PS-Lip}, and since \(\nabla u=0\) a.e.\ on \(\{u=0\}\), for every \(p\in(1,+\infty)\), we have
\begin{equation}
\label{eq:gradients-bounds}
\|\nabla u^\sharp\|_{p,\mu,X}
\le
\|\nabla u\|_{p,\mu,X}
\le
\|\nabla u \|_{L^\infty(X)} \,m^{\frac 1p}.
\end{equation}
Let now \(s<\|\nabla u^\sharp\|_{L^\infty(X)}\), where the essential supremum is equivalently taken with respect to \(\mu\) or \(\mathcal{L}^N\), these measures being mutually absolutely continuous. 
Then, by definition of essential supremum, \(S=\{|\nabla u^\sharp|>s\}\) has positive \(\mu\)-measure, so that by~\eqref{eq:gradients-bounds} and by trivial lower estimates, for all \(p\in(1,+\infty)\), we get
\[
\|\nabla u \|_{L^\infty(X)} m^{\frac 1p}
\ge
\|\nabla u^\sharp\|_{p,\mu,X}
\ge
s \mu(S)^{\frac 1p}.
\]
In particular, \(\mu(S)<+\infty\) whenever \(s>0\). 
Letting now \(p\to +\infty\), we get \(s\le \|\nabla u \|_{L^\infty(X)}\), and the claimed bound follows by arbitrariness of \(s\).

Assume now that \(X\) is convex, and let \(x,y\in X\). 
The segment \([x,y]\) is a compact subset of the open set \(X\), thus it has positive distance from \(\partial X\). 
Since \(u^\sharp\) is locally Lipschitz by \cref{thm:talenti}, its distributional gradient coincides with its a.e.\ gradient, and, for \(\varepsilon\) small enough, the mollifications \(u^\sharp*\rho_\varepsilon\) are well-defined and smooth in a neighborhood of \([x,y]\), where they satisfy
\[
|\nabla(u^\sharp*\rho_\varepsilon)|
=
|(\nabla u^\sharp)*\rho_\varepsilon|
\le
\|\nabla u^\sharp\|_{L^\infty(X)}
\le
\|\nabla u\|_{L^\infty(X)},
\]
by the first part of the statement. 
Moreover, \(u^\sharp*\rho_\varepsilon\to u^\sharp\) uniformly on \([x,y]\). 
Hence,
\[
|u^\sharp(x)-u^\sharp(y)|
=
\lim_{\varepsilon\to0^+}|u^\sharp*\rho_\varepsilon(x)-u^\sharp*\rho_\varepsilon(y)|
\le
\|\nabla u\|_{L^\infty(X)}\,|x-y|.
\qedhere
\]
\end{proof}

\section{Structure of admissible \texorpdfstring{\(f\)}{f}}
\label{sec:structure}
In this section, assuming convexity of the ambient space \(X\), we establish some structural properties of the functions \(f\) satisfying~\ref{hp:N1}--\ref{hp:N3}, \ref{hp:T3}, and~\ref{hp:T4}. 
Specifically, in \cref{thm:geo-uc} we prove a quantitative separation of their level sets,
which is the key ingredient in the proof of the P\'olya--Szeg\H{o} principle below, see \cref{thm:polya-szego}. 
Further structural properties of \(f\), which are not needed here, are studied in our follow-up paper~\cite{BS-geom}.

\begin{theorem}
\label{thm:geo-uc}
Let \(X\) be convex, and assume~\ref{hp:N1}--\ref{hp:N3}, \ref{hp:T3}, and~\ref{hp:T4}.
Then, the following hold.
\begin{enumerate}[label=(\alph*)]
\item
\label{item:g} 
There exists a continuous function \(g:(0,1)\to(0,+\infty)\) such that \(\|\nabla f(x)\|=g(f(x))\) for all \(x\in X_{\mathrm{reg}}\).
\item
\label{item:separation}
Let \(c_0\in(0,1)\) and
\[
G(c)=\int_{c_0}^c \frac{\mathrm{d}t}{g(t)}, \qquad
\text{for } c\in(0,1).
\]
If \(y_1,y_2\in X\) and \(0<\tau_1<\tau_2<1\) are such that \(f(y_1)\le\tau_1\) and \(f(y_2)\ge\tau_2\), then
\begin{equation}
\label{eq:separation}
\|y_1-y_2\| \ge G(\tau_2)-G(\tau_1) >0.
\end{equation}
In particular,
\[
\dist(\{f\le t_1\};\{f\ge t_2\})
>
0
\quad
\text{whenever } 
0\le t_1<t_2\le 1.
\]
\end{enumerate}
\end{theorem}

\begin{proof}
We divide the proof into several steps. 
By~\ref{hp:N2}, \(\nabla f\neq 0\) on \(X_{\mathrm{reg}}\), thus each level set \(M_c\), \(c\in(0,1)\), is a smooth hypersurface, connected by~\ref{hp:T3}.

\medskip
\textit{Proof of~\ref{item:g}.}
This is mentioned en passant in~\cite[Property~(viii)]{Tal97}; though, the key request of level sets being connected is missing there. 
We give here the proof.
Let \(\mathbf{N}\) be the unit normal vector field to the level sets of \(f\) in \(X_{\mathrm{reg}}\), defined in~\eqref{eq:def-N}.
A trajectory of steepest ascent \(\gamma(s)\) parameterized by arclength \(s\) satisfies the ODE \(\gamma'(s) = \mathbf{N}(\gamma(s))\). 
Since, by~\ref{hp:T4}, these trajectories are straight lines, \(\mathbf{N}(\gamma(s))\) is constant along each of them, giving \(\gamma(s) = x_0 + s \mathbf{N}(x_0)\), where \(x_0=\gamma(0)\) is the initial datum. 
By constancy of \(\mathbf{N}\), differentiating we get
\begin{equation}
\label{eq:straight-D2}
(D\mathbf{N}) \mathbf{N}
=
0.
\end{equation}
For \(x\in X_{\mathrm{reg}}\), we let \(w(x) = \|\nabla f(x)\| > 0\). 
Thus, by definition,
\begin{equation}
\label{eq:nablaf-w}
\mathbf{N}(x) 
=
\frac{\nabla f(x)}{w(x)}.
\end{equation}
Differentiating the above equality, we find 
\[
D\mathbf{N}(x)
=
\frac 1w D\nabla f(x) - \frac{\nabla f(x) \otimes \nabla w(x)}{w^2}
=
\frac 1w D^2 f(x) - \frac 1w \mathbf{N}(x)\otimes \nabla w(x).
\]
Applying the above to \(\mathbf{N}(x)\), using~\eqref{eq:straight-D2}, rearranging, and owing to tensor algebra rules (\((a\otimes b)c = (b\cdot c)a\)), we find
\[
D^2f\,\mathbf{N} 
=
(\mathbf{N} \otimes \nabla w)\mathbf{N}
=
(\nabla w\cdot \mathbf{N})\mathbf{N}.
\]
On the other hand, differentiating the equality \(w^2 = \|\nabla f\|^2\), rearranging, and using the relation~\eqref{eq:nablaf-w}, we get \(\nabla w = D^2f\, \mathbf{N}\).
Putting these two last equalities together and recalling the definition of \(w\) and of \(\mathbf{N}\), we have
\[
\nabla \|\nabla f\|
=
\left(\nabla\|\nabla f\| \cdot \frac{\nabla f}{\|\nabla f\|}\right) \frac{\nabla f}{\|\nabla f\|},
\]
that is, \(\nabla \|\nabla f\|\) is parallel to \(\nabla f\). 
Hence, the tangential gradient of \(\|\nabla f\|\) along each level set \(M_c\), \(c\in(0,1)\), vanishes, so that \(\|\nabla f\|\) is constant on each connected component of \(M_c\). 
Since the level sets are connected by~\ref{hp:T3}, one has
\begin{equation}
\label{eq:g(f)}
\|\nabla f(x)\| = g(f(x)),
\qquad \forall\, x\in X_{\mathrm{reg}},
\end{equation}
for some function \(g:(0,1)\to (0,+\infty)\).
To see that \(g\) is continuous, we first notice that \(f\) maps the open set \(X_{\mathrm{reg}}\) onto \((0,1)\), since \(X\) is connected, \(\inf f=0\), and \(\sup f=1\). 
Moreover, \(f\) is an open map on \(X_{\mathrm{reg}}\). 
Indeed, given \(x\in X_{\mathrm{reg}}\) and \(r>0\) with \(B_r(x)\subset X_{\mathrm{reg}}\), the continuous function
\[
\phi(\varepsilon)=f(x+\varepsilon\nabla f(x))
\]
satisfies \(\phi'(0)=|\nabla f(x)|^2>0\), so that it is strictly increasing on \([-\delta,\delta]\) for some \(0<\delta\ll1\). 
Up to choosing a smaller \(\delta\), the segment \(\{x+\varepsilon\nabla f(x)\,:\,|\varepsilon|\le\delta\}\) is contained in \(B_r(x)\). 
Then, by the intermediate value theorem, we get
\[
f(B_r(x)) 
\supset 
\{f(x+\varepsilon \nabla f(x))\,:\, |\varepsilon|\le\delta\}
=
[\phi(-\delta),\phi(\delta)]
\supset
(\phi(-\delta), \phi(\delta)).
\]
Therefore, for every open set \(U\subset(0,+\infty)\), owing to~\eqref{eq:g(f)} and to the surjectivity of \(f\), we have
\[
g^{-1}(U)
=
f\bigl(\{\,x\in X_{\mathrm{reg}}\,:\,\|\nabla f(x)\|\in U\,\}\bigr),
\]
which is open, being the image under the open map \(f\) of an open set, as \(x\mapsto\|\nabla f (x)\|\) is continuous. 
Thus, \(g\) is continuous.

\medskip
\textit{Proof of~\ref{item:separation}.}
%
Let \(y_1\), \(y_2\), \(\tau_1\), and \(\tau_2\) be as in~\ref{item:separation}. 
Clearly \(y_1\neq y_2\); let
\[
\ell=\|y_2-y_1\|
\quad
\text{and }
e=\frac{y_2-y_1}{\ell}.
\]
Since \(X\) is convex, the function \(\phi(\tau)=f(y_1+\tau e)\) is well-defined; since \(f\) is smooth it also is continuous on \([0,\ell]\). 
Since, by assumption,
\[
\phi(0)
=
f(y_1)
\le
\tau_1
<
\tau_2
\le
f(y_2)
=
\phi(\ell),
\]
the numbers
\[
\sigma_2=\min\{\,\tau\in[0,\ell]\,:\,\phi(\tau)=\tau_2\,\},
\quad
\sigma_1=\max\{\,\tau\in[0,\sigma_2]\,:\,\phi(\tau)=\tau_1\,\},
\]
are well-defined, and by the intermediate value theorem \(\tau_1<\phi<\tau_2\) on \((\sigma_1,\sigma_2)\). 
Hence, the segment \(\{y_1+\tau e\,:\,\tau\in[\sigma_1,\sigma_2]\}\) lies in \(X_{\mathrm{reg}}\). 
Moreover, the function \(\phi\) is \(C^1\) on \([\sigma_1,\sigma_2]\) and, by~\eqref{eq:g(f)}, 
\[
|\phi'(\tau)|\le\|\nabla f(y_1+\tau e)\|=g(\phi(\tau)).
\]
Therefore, by definition of \(\ell\), by definition of \(\sigma_i\), by definition of \(G\), the above inequality, and the chain rule on \((G\circ \phi)'\) which shows that  \((G\circ \phi)'\le 1\) on \([\sigma_1,\sigma_2]\), once again by definition of \(\sigma_i\), and finally by definition of \(G\) and the positivity of \(g\), we have
\begin{align*}
\| y_2 - y_1\|
=
\ell
\ge
\sigma_2 -\sigma_1
&
\ge
\int_{\sigma_1}^{\sigma_2}(G\circ\phi)'(\tau)\,\mathrm{d}\tau
\\
&
=
G(\tau_2)-G(\tau_1)
=
\int_{\tau_1}^{\tau_2}\frac{1}{g(t)}\,\mathrm{d}t 
>
0,
\end{align*}
which proves~\eqref{eq:separation}. 

Let now \(0\le t_1<t_2\le 1\) be fixed  and let \(y_1,y_2\) be any two points in, respectively, \(\{f\le t_1\}\) and \(\{f\ge t_2\}\). 
Let us define
\[
\tau_1
=
\frac{2t_1+t_2}{3}
\qquad
\text{and}
\qquad
\tau_2
=
\frac{t_1+2t_2}{3},
\]
which satisfy \(0\le t_1 < \tau_1 < \tau_2 < t_2 \le 1\). 
Applying~\eqref{eq:separation} with \(\tau_1, \tau_2\) as above, and by the arbitrariness of \(y_1,y_2\), we get the quantitative separation of the level sets.
\end{proof}

\section{P\'olya--Szeg\H{o} principle in \texorpdfstring{\(W^{1,p}_{\mu,0}(\Omega;X)\)}{W1p mu0 (Omega)}}
\label{sec:polya-szego}

We here prove that~\eqref{eq:PS-Lip} extends from Lipschitz functions on $X$ to the Sobolev space $W^{1,p}_{\mu,0}(\Omega;X)$, via a standard density argument.

\begin{theorem}[P\'olya--Szeg\H{o} principle in $W^{1,p}_{\mu,0}(\Omega;X)$]
\label{thm:polya-szego}
Let \(X\) be convex, assume~\ref{hp:N1}--\ref{hp:N3} and~\ref{hp:T1}--\ref{hp:T4}, let \(\Omega\subsetneq X\) be a nonempty open set with \(\mu(\Omega)<+\infty\), and let \(p\in(1,+\infty)\).
Then, for every $u \in W^{1,p}_{\mu,0}(\Om;X)$, the Talenti symmetrization $u^\sharp$  belongs to $W^{1,p}_{\mu,0}(\Om^\sharp;X )$ and
\begin{equation}
\label{eq:extended-ps}
     \int_\Om |\nabla u|^p \, \di \mu 
    \geq \int_{\Om^\sharp} |\nabla u^\sharp|^p \, \di \mu.
\end{equation}
\end{theorem}

\begin{proof}
For the sake of notation, we write \(v\) in place of \(u\), as the letter \(u\) is booked below for auxiliary functions. 
Let \(v\in W^{1,p}_{\mu,0}(\Omega;X)\). If \(v\equiv0\), there is nothing to prove, thus we assume \(v\not\equiv0\). 
Recall that \(v^\sharp=|v|^\sharp\) is defined on the whole \(X\); by \cref{lem:levelset-symmetrizedset-inclusion}, it vanishes outside \(\Omega^\sharp\); and by \cref{lem:sym-set}, \(\Omega^\sharp=\{f>t_\Omega\}\), with \(t_\Omega\in[0,1)\) and \(\mu(\Omega^\sharp)=\mu(\Omega)\).

\medskip
\textit{Step~1: Approximation via Lipschitz functions.}
By definition, there exists a sequence \((v_n)_{n\in\N}\subset\mathrm{Lip}_0(\Omega;X)\) such that \(v_n\to v\) in the norm~\eqref{eq:norm}; without loss of generality, \(v_n\not\equiv0\) for all \(n\). 
Each \(v_n\) is Lipschitz on \(X\), vanishes on \(X\setminus\Omega\), has bounded support, and
\[
d_n = \dist(\supp_X(v_n); \partial \Omega\cap X) > 0.
\]

\medskip
\textit{Step~2: Symmetrization of the sequence.}
Since \(\mu(\{v_n\neq0\})\le\mu(\Omega)<+\infty\), \cref{thm:talenti} applies to \(v_n\), and \(\nabla v_n=0\) a.e.\ on \(X\setminus\Omega\), so that
\begin{equation}\label{eq: dis norme gradiente}
\int_{\Omega} |\nabla {v}_n|^p\di\mu
=
\int_{X} |\nabla {v}_n|^p\di\mu
\geq
\int_{X}|\nabla {v}_n^\sharp|^p\di\mu
\ge
\int_{\Omega^\sharp}|\nabla {v}_n^\sharp|^p\di\mu.
\end{equation}
Together with~\eqref{eq:Lp-norm-preserved}, this gives \(v_n^\sharp\in W^{1,p}_\mu(\Omega^\sharp)\). 
Moreover, since \(X\) is convex, by \cref{lem:lip-bound}, \(v_n^\sharp\) is Lipschitz on \(X\). 

\medskip
\textit{Step~3: Boundary buffer layer.}
We claim that
\begin{equation}
\label{eq:dn}
d_n^\sharp
=
\dist(\supp_X(v_n^\sharp); X\setminus\Omega^\sharp) 
> 
0.
\end{equation}
Since \(\Omega\) is a nonempty proper open subset of the connected set \(X\), there exists \(b\in\partial\Omega\cap X\). Let \(r_n\in(0,d_n)\) be such that \(B_{r_n}(b)\subset X\). 
The set \(U_n=B_{r_n}(b)\cap\Omega\) is open and nonempty, since \(b\in\partial\Omega\), and it is disjoint from \(\supp_X(v_n)\), since its points have distance less than \(d_n\) from \(\partial\Omega\cap X\). 
Since \(\mu\) is strictly positive, \(\mu(U_n)>0\), and therefore
\[
m_n
=
\mu(\{\,x\,:\, v_n(x) \neq 0\,\})
\le
\mu(\Omega\setminus U_n) 
=
\mu(\Omega) - \mu(U_n) 
<
\mu(\Omega).
\]
By \cref{lem:Phi-f}, \(\tau_n=\Phi_f^{-1}(m_n)\) satisfies \(\tau_n>t_\Omega\), because \(\Phi_f\) is strictly decreasing and \(\Phi_f(t_\Omega)=\mu(\Omega)>m_n\). 
By the characterization of the zero set of \(v_n^*\) given after \cref{def:d-rearrangement}, \(v_n^*(s)>0\) if and only if \(s<m_n\); hence,
\[
\{\,v_n^\sharp\neq0\,\}
=
\{\,x\in X\,:\,\Phi_f(f(x))<m_n\,\}
=
\{\,f>\tau_n\,\},
\]
so that \(\supp_X(v_n^\sharp)\subset\{f\ge\tau_n\}\), by continuity of \(f\). 
Since \(X\setminus\Omega^\sharp=\{f\le t_\Omega\}\) and \(t_\Omega<\tau_n\), \cref{thm:geo-uc}\ref{item:separation} gives~\eqref{eq:dn}. In particular, \(\supp_X(v_n^\sharp)\subset\Omega^\sharp\) and, as \(\Omega^\sharp\) is open, so that \(\partial\Omega^\sharp\cap X\subset X\setminus\Omega^\sharp\), we have
\[
\dist(\supp_X(v_n^\sharp);\partial\Omega^\sharp\cap X)
\ge
d_n^\sharp>0.
\]

\medskip
\textit{Step~4: The symmetrized functions belong to \(W^{1,p}_{\mu,0}(\Omega^\sharp;X)\).}
Let \(\chi\) be a smooth cutoff function on the half-line \([0,+\infty)\), with \(0\le\chi\le1\), \(\chi(t)=1\) on \([0,1]\), \(\chi(t)=0\) for \(t\ge 2\), and \(|\chi'(t)| \le 2\), and let \(\chi_R(x)=\chi(\|x\|/R)\). 
The functions
\[
u_R^n
=
v^\sharp_n \chi_R
\]
are Lipschitz on \(X\), as \(v^\sharp_n\) is Lipschitz and bounded, have bounded support as \(\chi_R\) has, and \(\supp_X(u^n_R)\subset\supp_X(v^\sharp_n)\). 
By Step~3, also owing to the fact that \(\mu(\{u^n_R\neq0\})\le\mu(\{v^\sharp_n\neq0\})=m_n<+\infty\), \(u_R^n\in\mathrm{Lip}_0(\Omega^\sharp;X)\). 
Moreover,
\[
|u_R^n-v_n^\sharp|\le(1-\chi_R)|v_n^\sharp|,
\qquad
|\nabla u_R^n-\nabla v_n^\sharp|\le(1-\chi_R)|\nabla v_n^\sharp|+\frac{2}{R}|v_n^\sharp|,
\]
so that, by the Dominated Convergence Theorem, \(u_R^n\to v_n^\sharp\) in \(W^{1,p}_\mu(\Omega^\sharp)\), as \(R\to+\infty\). 
Hence, \(v^\sharp_n\in W^{1,p}_{\mu,0}(\Omega^\sharp;X)\).

\medskip
\textit{Step~5: Passing to the limit.}
By~\eqref{eq:lp-non-expansive}, applied to the nonnegative functions \(|v_n|\) and \(|v|\), and since \(\bigl||v_n|-|v|\bigr|\le|v_n-v|\), we have
\[
\|v_n^\sharp-v^\sharp\|_{p,\mu,\Omega^\sharp}
=
\bigl\||v_n|^\sharp-|v|^\sharp\bigr\|_{p,\mu,\Omega^\sharp}
\le
\bigl\||v_n|-|v|\bigr\|_{p,\mu,\Omega}
\le
\|v_n-v\|_{p,\mu,\Omega}
\to0.
\]
Moreover, the gradients \(\nabla v_n^\sharp\) are uniformly bounded in \(L^p_\mu(\Omega^\sharp)\), as a consequence of~\eqref{eq: dis norme gradiente} together with the convergence of \((v_n)_n\) in the norm~\eqref{eq:norm}. 
Hence, \((v_n^\sharp)_{n \in \N}\) is bounded in \(W^{1,p}_{\mu,0}(\Omega^\sharp;X)\).

Since \(W^{1,p}_{\mu,0}(\Omega^\sharp;X)\) is reflexive, by the Eberlein--\v{S}mulian Theorem there exist a subsequence \((v_{n_k}^\sharp)_{k\in\N}\) and a limit function \(V \in W^{1,p}_{\mu,0}(\Omega^\sharp;X)\) (which is weakly closed, being a closed subspace) such that
\[
v_{n_k}^\sharp\rightharpoonup V 
\quad
\text{weakly in } 
L^p_\mu(\Omega^\sharp)
\quad 
\text{and} 
\quad 
\nabla v_{n_k}^\sharp\rightharpoonup\nabla V
\quad \text{weakly in }
L^p_\mu(\Omega^\sharp).
\]
From the uniqueness of the limit, \(V=v^\sharp\), \(\mu\)-almost everywhere; therefore, \(v^\sharp \in W^{1,p}_{\mu,0}(\Omega^\sharp;X)\). 
Further, by the weak convergence of the gradients, by inequality~\eqref{eq: dis norme gradiente}, and by the convergence of \((v_n)_{n\in\N}\) in \(W^{1,p}_{\mu}(\Omega)\), one has
\[
\|\nabla v^\sharp\|_{p,\mu,\Omega^\sharp}
\leq 
\liminf_{k \to \infty}\|\nabla v_{n_k}^\sharp\|_{p,\mu,\Omega^\sharp}
\leq
\liminf_{k\to \infty}\|\nabla v_{n_k}\|_{p,\mu,\Omega}
=
\|\nabla v\|_{p,\mu,\Omega}.
\qedhere
\]
\end{proof}

As usual, the P\'olya--Szeg\H{o} principle on Sobolev spaces implies the corresponding Faber--Krahn and Saint-Venant inequalities.

\begin{corollary}[Faber--Krahn inequality]
\label{cor:fk-ineq}
Let \(X\) be convex, assume~\ref{hp:N1}--\ref{hp:N3} and~\ref{hp:T1}--\ref{hp:T4}, let \(\Omega\subsetneq X\) be a nonempty open set with \(\mu(\Omega)<+\infty\), and let \(p\in(1,+\infty)\).
Then,
\begin{equation*}
    \lambda^\mu_{1,p}(\Om^\sharp)\leq \lambda_{1,p}^\mu(\Om).
\end{equation*}
\end{corollary}

\begin{proof}
Let $u \in W^{1,p}_{\mu,0}(\Om;X)$. By \cref{thm:polya-szego}, the Talenti symmetrization $u^\sharp$ belongs to $ W^{1,p}_{\mu,0}(\Om^\sharp; X)$. 
Moreover, by~\eqref{eq:Lp-norm-preserved} and~\eqref{eq:extended-ps}, we have
\[
\lambda^\mu_{1,p}(\Om^\sharp)
\leq 
\frac{\displaystyle\int_{\Om^\sharp}|\nabla u^\sharp|^p\di \mu}{\displaystyle\int_{\Om^\sharp}|u^\sharp|^p\di\mu}
\leq 
\frac{\displaystyle\int_\Om|\nabla u|^p\di \mu}{\displaystyle\int_\Om|u|^p\di\mu}.
\]
Taking now the infimum on the right-hand side among non identically zero functions in $W^{1,p}_{\mu,0}(\Om;X)$, the claim follows.
\end{proof}

By \cref{lem:sym-set}, \(\Omega^\sharp=\{f>t_\Omega\}\) depends on \(\Omega\) only through \(\mu(\Omega)\), and \(\mu(\Omega^\sharp)=\mu(\Omega)\). 
Hence, \cref{cor:fk-ineq} is precisely the Faber--Krahn inequality~\eqref{eq:fk-ineq}, with \(\Omega^*\) the superlevel set of \(f\) having measure \(\mu(\Omega)\); analogously, \cref{cor:sv-ineq} below is the Saint-Venant inequality~\eqref{eq:sv-ineq}.

\begin{corollary}[Saint-Venant inequality]
\label{cor:sv-ineq}
Let \(X\) be convex, assume~\ref{hp:N1}--\ref{hp:N3} and~\ref{hp:T1}--\ref{hp:T4}, let \(\Omega\subsetneq X\) be a nonempty open set with \(\mu(\Omega)<+\infty\), and let \(p\in(1,+\infty)\).
Then,
\begin{equation*}
T^\mu_{p}(\Om^\sharp)
\geq 
T_{p}^\mu(\Om).
\end{equation*}
\end{corollary}

\begin{proof}
Let \(u\in W^{1,p}_{\mu,0}(\Om;X)\) with \(\nabla u\neq0\); since we are assuming \(\mu(\Omega)<+\infty\), we have \(u\in L^1_\mu(\Omega)\). 
By \cref{thm:polya-szego}, \(u^\sharp\in W^{1,p}_{\mu,0}(\Om^\sharp;X)\) and, by the proof of \cref{prop:cnt-sym}, which holds verbatim for \(p=1\), \(\|u^\sharp\|_{1,\mu,\Omega^\sharp}=\|u\|_{1,\mu,\Omega}>0\). 
If \(\nabla u^\sharp=0\), then \(T^\mu_p(\Om^\sharp)=+\infty\): indeed, picking any \(\psi\in\mathrm{Lip}_0(\Omega^\sharp;X)\) with \(\nabla\psi\neq0\), the functions \(u^\sharp+\varepsilon\psi\) are admissible and their ratios diverge as \(\varepsilon\to0^+\). 
Otherwise, by~\eqref{eq:extended-ps},
\[
T^\mu_{p}(\Om^\sharp)
\geq 
\frac{\left(\displaystyle\int_{\Om^\sharp}|u^\sharp|\di \mu\right)^{p}}{\displaystyle\int_{\Om^\sharp}|\nabla u^\sharp|^p\di \mu}
\geq
\frac{\left(\displaystyle\int_\Om|u|\di \mu\right)^{p}}{\displaystyle\int_\Om|\nabla u|^p\di \mu}.
\]
Taking the supremum over all such \(u\), the claim follows.
\end{proof}

\begin{remark}
\label{rem:pq-eingevalue}
\cref{cor:fk-ineq,cor:sv-ineq} are particular cases of the more general Faber--Krahn-type inequality that one can prove on the first Dirichlet $(p,q)$-eigenvalue of the weighted \(p\)-Laplacian on $\Omega$, $\lambda^\mu_{1,p,q}(\Omega)$. 
This is defined as
\[
\lambda^\mu_{1,p,q}(\Omega)
=
\inf
\left\{
\frac{\|\nabla u\|^p_{p,\mu,\Omega}}{\|u\|^{p}_{q,\mu,\Omega}} \,:\, u\in W^{1,p}_{\mu,0}(\Omega;X),\, u \neq 0
\right\},
\]
where $q$ is chosen so that the continuous embedding $W^{1,p}_{\mu,0}(\Omega;X) \hookrightarrow L^q_\mu(\Omega)$ holds. 
These first $(p,q)$-eigenvalues are linked to the optimal Poincar\'e--Sobolev constant in the above-mentioned embedding, refer to~\cite{Bra14} in the Euclidean case.

The Talenti symmetrization keeps unchanged not only the $L^p_\mu$-norm of $u$, but rather all of its $L^q_\mu$-norms, as can be seen by following exactly the same proof as that of \cref{prop:cnt-sym}. 
Then, the P\'olya--Szeg\H{o} principle of \cref{thm:polya-szego} entails, as in the proof of \cref{cor:fk-ineq}, that
\[
\lambda^\mu_{1,p,q}(\Omega) 
\ge
\lambda^\mu_{1,p,q}(\Omega^\sharp).
\]
In particular, \cref{cor:fk-ineq} corresponds to the case $q=p$, whereas \cref{cor:sv-ineq} corresponds to the case $q=1$, since $T_p^\mu(\Omega) = 1/\lambda_{1,p,1}^\mu(\Omega)$, provided that nonzero constants do not belong to \(W^{1,p}_{\mu,0}(\Omega;X)\), so that the admissible classes in the two definitions coincide.

\end{remark}

\begin{remark}[Isoperimetric sets known only for small volumes]
\label{rem:threshold}
In several situations, isoperimetric sets are known only for volumes below a certain threshold. 
The previous results are robust with respect to this restriction. 
Precisely, let \(\overline{m}\in(0,\mu(X))\), assume~\ref{hp:N1}--\ref{hp:N3}, \ref{hp:T1'}--\ref{hp:T3'}, and~\ref{hp:T4}, and let \(\bar t=\Phi_f^{-1}(\overline{m})\in(0,1)\), so that, by \cref{lem:Phi-f}, \ref{hp:T2'} and~\ref{hp:T3'} concern the levels \(t\in[\bar t,1)\) and \(t\in(\bar t,1)\), respectively.
Then, \eqref{eq:PS-Lip} holds for every bounded Lipschitz function \(u:X\to\R\) satisfying the request \(\mu(\{u\neq0\})\le \overline{m}\). 
Indeed, in the proof of~\cite{Tal97}, the isoperimetric inequality and its saturation only enter through the comparison, for a.e.\ \(t>0\), between \(\Per_\mu(\{|u|>t\};X)\) and \(\Per_\mu(\{u^\sharp>t\};X)=\Per_\mu(\{f>\Phi_f^{-1}(\Phi_u(t))\};X)\). 
Since, by assumption, \(\Phi_u(t)\le\mu(\{u\neq0\})\le \overline{m}\) for all \(t>0\), only sets of measure at most \(\overline{m}\) and superlevel sets \(\{f>\tau\}\) with \(\tau\ge\bar t\) are involved. 
Similarly, the connectedness of the level sets of \(f\) is only needed for \(t\in(\bar t,1)\). 
Indeed, it only enters through \cref{thm:geo-uc}, whose proof works level by level: under~\ref{hp:T3'}, the proof of \cref{thm:geo-uc}\ref{item:g} yields~\eqref{eq:g(f)} on \(\{f>\bar t\}\), with \(g\) continuous on \((\bar t,1)\), and the proof of \cref{thm:geo-uc}\ref{item:separation}, with \(c_0\in(\bar t,1)\), only involves the values of \(g\) on \((\tau_1,\tau_2)\), so that~\eqref{eq:separation} holds for \(\bar t<\tau_1<\tau_2<1\), whence \(\dist(\{f\le t_1\};\{f\ge t_2\})>0\) for \(\bar t\le t_1<t_2\le1\).
This is all that is needed: in Talenti's argument, \(|\nabla f|\) has to be constant only on the level sets of \(f\) on which \(u^\sharp\) is not constant, and \(u^\sharp\) vanishes on \(\{f\le\Phi_f^{-1}(\mu(\{u\neq0\}))\}\supset\{f\le\bar t\}\), while, in Step~3 of the proof of \cref{thm:polya-szego}, \cref{thm:geo-uc}\ref{item:separation} is applied with \(\bar t\le t_\Omega<\tau_n\), as \(\mu(\Omega)\le\overline{m}\).

Consequently, \cref{lem:lip-bound} holds for such functions \(u\), and \cref{thm:polya-szego}, \cref{cor:fk-ineq,cor:sv-ineq}, and \cref{rem:pq-eingevalue} hold, with the same proofs, for every open set \(\Omega\) as therein with \(\mu(\Omega)\le \overline{m}\), since the approximating functions \(v_n\) in the proof of \cref{thm:polya-szego} satisfy \(\mu(\{v_n\neq0\})<\mu(\Omega)\le \overline{m}\). 
In other words, one obtains Faber--Krahn and Saint-Venant inequalities among sets of small measure:
\[
\lambda^\mu_{1,p}(\Omega^\sharp)
\le
\lambda^\mu_{1,p}(\Omega)
\qquad
\text{and}
\qquad
T^\mu_p(\Omega^\sharp)
\ge
T^\mu_p(\Omega),
\]
for all open sets \(\Omega\subsetneq X\) with \(\mu(\Omega)\le \overline{m}\), where \(\Omega^\sharp\) is the isoperimetric superlevel set of \(f\) with \(\mu(\Omega^\sharp)=\mu(\Omega)\).
\end{remark}

\section{Examples}
\label{sec:examples}

In this section we collect several examples that fall within the general framework presented, thus obtaining the corresponding Faber--Krahn and Saint-Venant inequalities.
Actually, in all of the following settings, the more general inequality for the first \((p,q)\)-eigenvalue of the weighted \(p\)-Laplacian mentioned in \cref{rem:pq-eingevalue} holds. 
Whenever no reference is given for one of these inequalities, it is, to the best of our knowledge, new.

In all the examples below the isoperimetric sets are: either balls, see \cref{sec:balls}, possibly known only for small volumes as in \cref{sec:small-volumes}; or half-spaces, see \cref{sec:half-spaces}; or complements of balls, see \cref{sec:complements}. As shown in the forthcoming paper~\cite{BS-geom}, this is essentially unavoidable.

\subsection{Isoperimetric sets are concentric balls}
\label{sec:balls}

Let $X\subset\mathbb{R}^N$ with \(N\ge 2\) be an open convex cone centered in the origin and let $\mu$ be a locally finite measure, that is, such that \(\mu(B_r(0)\cap X)<+\infty\) for all \(r>0\). 
Assume that the one-para\-met\-er family of concentric balls $\{\,B_r(0)\,:\,r>0\,\}$ is such that the family
\begin{equation}
\label{eq:balls-iso}
\{\,B_r(0) \cap X\,\}_r,
\end{equation}
is a complete family of isoperimetric sets, that is, for all $m\in(0,\mu(X))$ there exists $r=r(m)>0$ such that $\mu(B_r(0)\cap X)=m$ and
\begin{equation}
\label{eq:iso-balls}
\Per_\mu (B_r(0)\cap X; X) \le \Per_\mu(\Omega; X),
\qquad 
\forall \Omega\subset X
\text{ with $\mu(\Omega)=m$,}
\end{equation}
that is, condition~\ref{hp:T1} holds with
\[
I(m) 
=
\Per_\mu(B_{r(m)}(0)\cap X; X),
\quad 
m\in(0,\mu(X)),
\qquad 
I(0)=0,
\]
provided \(I\) is continuous on \([0,\mu(X))\); its positivity on \((0,\mu(X))\) is automatic, as \(\partial B_r(0)\cap X\) is a nonempty relatively open subset of the sphere \(\partial B_r(0)\) and \(e^W>0\). 
In all the examples below, \(I\) is explicit and its continuity can be checked directly.

We now let
\[
f(x) 
=
\frac{1}{1+\|x\|^2},
\qquad 
\forall x\in X.
\]
Trivially~\ref{hp:N1} holds: $\inf f = 0$; $\sup f = 1$; and $\mathcal{L}^N\bigl(f^{-1}(0)\bigr)=\mathcal{L}^N\bigl(f^{-1}(1)\bigr)=0$. 
Further, this function can be thought of as the composition $f = h(g)$, where
\[
h(t) 
=
\frac{1}{1+t}
\qquad
\text{and}
\qquad 
g(x)
=
\|x\|^2.
\]
Therefore, it is smooth and its gradient never vanishes since $\nabla f = h'(g)\nabla g$, with $h'<0$ on $(0,+\infty)$ and $\nabla g \neq 0$ outside the origin. 
Hence, \ref{hp:N2} is met. 
Moreover, it is immediate that all superlevel sets of $f$ are of the form in~\eqref{eq:balls-iso} so that: condition~\ref{hp:N3} holds, having asked \(\mu(B_r(0)\cap X)<+\infty\); condition~\ref{hp:T2} holds, having assumed the one-parameter family $\{\,B_r(0) \cap X\,\}_r$ to be isoperimetric; and condition~\ref{hp:T3} holds, since the level sets \(\{f=t\}= \partial B_r(0)\cap X\) are connected, as \(N\ge 2\) and \(X\) is convex.
Additionally, condition~\ref{hp:T4} is met, since the lines of steepest ascent are rays pointing toward the origin and they sweep the whole interval \((0,1)\), as $f$ is radial and decreasing in \(\|x\|\).

\subsubsection{Standard Lebesgue measure}
\label{sec:euclidean}

The standard Lebesgue measure falls under the above assumptions. 
It corresponds, up to a multiplicative prefactor, to the choices
\[
X=\mathbb{R}^N 
\qquad
\text{and}
\qquad
W(x)\equiv c 
\text{ with } 
c\in\mathbb{R}.
\]
The classic isoperimetric inequality has been the object of extensive studies in the 20th century, eventually fully solved by De Giorgi, refer for instance to the book~\cite{Mag12book} for a proof of~\eqref{eq:iso-balls} in this setting.

The corresponding Faber--Krahn inequality had been proved independently by Faber~\cite{Fab23} and Krahn~\cite{Kra25}, whereas the corresponding Saint-Venant inequality was proved by P\'olya~\cite{Pol48}, both relying on the \emph{Schwarz symmetrization}~\cite{HL30,Sch90book}. 
Such a symmetrization is a special case of the one of \cref{sec:symmetrization}.

\subsubsection{Radial log-convex measures}
\label{sec:log-convex}

A measure $\mu = e^W\mathcal{L}^N$ is said to be \emph{log-convex} if the potential $W$ is convex. 
A long-standing conjecture asserted that if $W$, besides being convex, is smooth and radial, then balls centered at the origin are isoperimetric sets, that is,~\eqref{eq:iso-balls} holds with $X=\mathbb{R}^N$. 
This has been positively settled in~\cite{Cha19}. 
Choosing then $X=\mathbb{R}^N$ and $W(x)$ satisfying the above properties yields the corresponding Faber--Krahn and Saint-Venant inequalities. 
The former has been recently proved in~\cite[Thm.~1.1]{CM25} and~\cite[Sect.~7.1]{NV25}.

\subsubsection{Convex cones and homogeneous densities}
\label{sec:CROS}

Let $X\subset \mathbb{R}^N$ be an open and convex cone with vertex at the origin. 
If the potential $W$ is constant, \emph{i.e.}, the measure $\mu$ is the Lebesgue measure up to a prefactor, the relative isoperimetric inequality~\eqref{eq:iso-balls} was first proved in~\cite{LP90}. 
Thus, the Faber--Krahn and Saint-Venant inequalities follow in this case. 

Relative weighted isoperimetric inequalities in open and convex cones with vertex at the origin have also been proved by Cabré, Ros-Oton, and Serra, see~\cite{CROS12} and \cite[Thm.~1.3]{CROS16}, under suitable assumptions on the density function $w=e^{W}$. 
Specifically, the density needs to be: continuous up to the boundary of $X$; positively homogeneous of degree $\alpha\ge 0$; and such that $w^{1/\alpha}$ is concave, if $\alpha>0$. 
Hence, the Faber--Krahn and Saint-Venant inequalities follow when the above assumptions are met.

In particular, we refer to~\cite[Sect.~2]{CROS16} for several choices of $X$ and of $w$ which fall within these hypotheses. 
We here only recall the case~\cite[Sect.~2(iv)]{CROS16}, \emph{i.e.}, 
\[
W(x) = \log\left(\prod_{i=1}^k x_i^{\alpha_i}\right),
\]
with $\alpha_i > 0$ and $X=\{\,x\in\mathbb{R}^N\,:\, x_i>0,\, \forall i=1,\dots,k\,\}$, which leads to the density
\[
w = e^W = \prod_{i=1}^k x_i^{\alpha_i}.
\]
When $k=N$, such a density is usually called a \emph{monomial weight}.

For monomial weights in the orthant, the isoperimetric inequality~\eqref{eq:iso-balls} was obtained simultaneously also by Brock, Chiacchio, and Mercaldo~\cite{BCM14} (as the case \(c=0\) of the weights of \cref{sec:anti-gauss}). 
The theorem of~\cite{CROS12,CROS16} is substantially more general for homogeneous weights in convex cones, but it does not contain the full result of~\cite{BCM14}, which also covers the nonhomogeneous weights which we discuss in \cref{sec:anti-gauss}. 
We also mention that a Faber--Krahn inequality was first proved in the case $N=2$ and $k=1$, for the standard Laplacian ($p=2$) in~\cite{MS81}.
To the best of our knowledge, neither the Faber--Krahn nor the Saint-Venant inequality is stated in~\cite{BCM14,CROS12,CROS16,CR13}. 
However, combining their isoperimetric inequalities with~\cite{Tal97} gives the P\'olya--Szeg\H{o} principle for Lipschitz functions with compact support in \(\overline{X}\), which is stated in~\cite[Prop.~4.2]{CR13} for monomial weights, and is used in~\cite[Sect.~1]{CROS16} to obtain sharp weighted Sobolev inequalities in convex cones. 
For \(k=1\), up to relabeling the coordinates, the P\'olya--Szeg\H{o} principle for \(p=2\) among functions vanishing on \(\partial\Omega\cap X\) is proved in~\cite[Thm.~4.1]{BCM12}, and the Faber--Krahn inequality for \(p=2\) can be deduced from it.

\subsubsection{Monomial and anti-Gaussian weights}
\label{sec:anti-gauss}

Consider the open orthant \(X=\{\,x\in\R^N\,:\,x_i>0,\ i=1,\dots,N\,\}\) and the weights
\begin{equation}
\label{eq:anti-gauss}
w(x)=e^{W(x)}=e^{c\|x\|^2}\prod_{i=1}^N x_i^{k_i},
\qquad
c\ge0,\quad k_1,\dots,k_N\ge0,
\end{equation}
so that
\[
W(x)=c\|x\|^2+\sum_{i=1}^Nk_i\log x_i
\]
is continuous on \(X\). 
For \(c=0\), these are the monomial weights of \cref{sec:CROS}; for \(c>0\), the factor \(e^{c\|x\|^2}\) is anti-Gaussian, homogeneity and scaling are lost, and \(\mu(X)=+\infty\), while \(\mu(B_r(0)\cap X)<+\infty\) for all \(r>0\).

The sharp isoperimetric inequality~\eqref{eq:iso-balls} in this setting has been proved in~\cite{BCM14}. 
We also refer to~\cite{BCM12} for the case
\[
w(x)=x_N^ke^{c\|x\|^2},
\]
with \(c,k\ge 0\), with \(X\) the half-space \(\{x_N>0\}\).

In all these cases, \(X\) is an open convex cone with vertex at the origin, \(\mu\) is finite on bounded sets, and the intersections of \(X\) with the balls centered at the origin are isoperimetric. 
Hence, the framework of this section applies, and, for every \(p\in(1,+\infty)\), \cref{thm:polya-szego}, \cref{cor:fk-ineq,cor:sv-ineq}, and \cref{rem:pq-eingevalue} hold, among open sets \(\Omega\subset X\) with \(\mu(\Omega)<+\infty\), the optimal sets being \(B_r(0)\cap X\). 
As for the spectral consequences, the Faber--Krahn inequality for \(p=2\) in the half-plane with \(c=0\) was proved in~\cite{MS81}, see \cref{sec:CROS}. 
In the half-space case, the P\'olya--Szeg\H{o} principle for \(p=2\) is proved in~\cite[Thm.~4.1]{BCM12}, from which the Faber--Krahn inequality for \(p=2\) can be deduced, and the Saint-Venant inequality for \(p=2\) can be deduced from the comparison result~\cite[Thm.~1.2]{BCM12}, choosing a constant datum. 
To the best of our knowledge, the other cases of the Faber--Krahn and Saint-Venant inequalities, of the P\'olya--Szeg\H{o} principle in \(W^{1,p}_{\mu,0}(\Omega;X)\), and the inequalities for the \((p,q)\)-eigenvalues have not been stated before.

\subsection{Isoperimetric sets are half-spaces}
\label{sec:half-spaces}
Suppose the ambient space $X$ is a product of nonempty open intervals \((a_i,b_i)\) with $a_i, b_i \in [-\infty, +\infty]$. 
This covers, for instance, open slabs, open half-spaces, or all of $\Rn$. Let $\mu$ be a finite measure on $X$ so that~\ref{hp:N3} is satisfied, independently of $f$. 
Given a direction $\nu \in \mathbb{S}^{N-1}$ we consider the one-parameter family of half-spaces $\{\,H_r(\nu)\,:\, r\in\mathbb{R}\,\}$, where
\[
H_r(\nu)
=
\{\, x \in \R^{N}\,:\,\langle x,\nu\rangle>r\,\}.
\]
Assume there exists a direction $\theta$ such that the family $\{H_r(\theta)\cap X\}_r$ is such that for all $m\in(0,\mu(X))$, there exists $r=r(m)$ with $\mu(H_r(\theta)\cap X)=m$ such that
\begin{equation}
\label{eq:iso-half-spaces}
\Per_\mu (H_r(\theta)\cap X; X) 
\le
\Per_\mu(\Omega; X),
\qquad 
\forall \Omega\subset X
\text{ with $\mu(\Omega)=m$,}
\end{equation}
that is, condition~\ref{hp:T1} holds with
\[
I(m) 
=
\Per_\mu(H_{r(m)}(\theta)\cap X; X),
\quad 
m\in(0,\mu(X)),
\qquad 
I(0)=0,
\]
provided \(I\) is continuous on \([0,\mu(X))\) and positive on \((0,\mu(X))\), which, again, can be checked directly in all the examples below, where \(I\) is explicit or the measure has a product structure.
We further require that \(\langle x,\theta\rangle\) takes all real values as \(x\) ranges in \(X\); this is automatic if \(X=\R^N\), while, if \(X=\R^{N-1}\times(a,b)\) is a slab or a half-space, it holds if and only if \(\theta\neq\pm e_N\), and in particular if \(\langle\theta,e_N\rangle=0\).

Given $h:\R\to (0,1)$ any smooth map with \(h'>0\) on \(\R\), $\inf h=0$, and $\sup h=1$, we let 
\begin{equation}\label{eq:def-f-halfspaces}
f_{\theta}(x)
=
h(\langle x,\theta\rangle), 
\qquad 
\forall x \in X.
\end{equation}
Trivially: $f_{\theta}$ is smooth; $\inf f_{\theta}=0$; $\sup f_{\theta} =1$ (owing to the assumption on \(\theta\)); and $\nabla f_{\theta} \neq 0$, since we have $\nabla f_{\theta}=h'(\langle x,\theta \rangle )\theta$ with $h'>0$ in $\R$. 
Therefore, assumptions~\ref{hp:N1} and~\ref{hp:N2} are met.

By construction, the superlevel sets of $f_{\theta}$ are precisely half-spaces orthogonal to $\theta$, intersected with $X$:
\[
\{\, x \in X\,:\,f_{\theta}(x)>t\,\}
=
\{\, x \in X\,:\,\langle x,\theta\rangle>h^{-1}(t)\,\}=H_{h^{-1}(t)}(\theta)\cap X.
\]
This, combined with our hypothesis on $\{H_r(\theta)\cap X\}_{r}$, guarantees that condition~\ref{hp:T2} is satisfied. 
The lines of steepest ascent are straight lines parallel to $\theta$, oriented as \(\theta\), hence also condition~\ref{hp:T4} holds. 
In each of the examples below, we specify the direction \(\theta\), which is the inner unit normal of the boundary hyperplane \(\{\langle x,\theta\rangle=r\}\) of the minimizing half-spaces \(H_r(\theta)\). 
Finally, \(X\) is convex, and the level sets of \(f_\theta\), being intersections of \(X\) with hyperplanes, are connected, so that~\ref{hp:T3} holds.

\subsubsection{Standard Gaussian measure}
\label{sec:gaussian}

The Gaussian setting is the first instance that comes to mind where half-spaces are the isoperimetric sets. 
It corresponds, up to a multiplicative prefactor, to the choices
\begin{equation}
\label{eq:std-gaussian}
X
=
\Rn
\qquad
\text{and}
\qquad 
W(x) 
= 
-\frac{\|x\|^2}{2},
\end{equation}
that is, the Euclidean space weighted by $e^{-\|x\|^2/2}$.

The celebrated Gaussian isoperimetric inequality, independently proved by Borell~\cite{Bor75} and Sudakov--Cirel'son~\cite{SC74}, guarantees the validity of~\eqref{eq:iso-half-spaces} for any choice of $\theta$.

The associated Faber--Krahn inequality has been proved in~\cite[Thm.~4.7]{Ehr84}, and extended in~\cite{BCF07} to a broader class of elliptic operators associated with the Gaussian measure, whereas the Saint-Venant inequality is proved in~\cite[Prop.~5.6]{Liv24} for $p=2$ (for general $p$ it is mentioned in~\cite[Thm.~2.3]{CFMW26}). 
Their proofs rely on the \emph{Ehrhard symmetrization}~\cite{Ehr83}. 
In our framework, Ehrhard's construction is exactly equivalent to the $(\mu,f_\theta)$-Talenti symmetrization on $X=\Rn$, provided $f_\theta$ is defined according to~\eqref{eq:def-f-halfspaces} for an arbitrary unit direction $\theta$.
For completeness, we mention that in~\cite[Cor.~1.5]{CFLS26} (for $p=2$) and~\cite[Cor.~1.4]{CQS26} (for general $p$), a Faber--Krahn inequality under a  \emph{mean width constraint} (rather than a measure constraint) has been proved in the Gaussian setting: balls centered at the origin are optimal.  
While seemingly surprising, note that the mean width is measure-independent (and that half-spaces have infinite mean width!).

\subsubsection{Gauss-like product measures}
\label{sec:gauss-like}

Gauss-like product measures are perturbations of the Gaussian measure on product domains, which have been studied in~\cite{BCM16}.
The ambient space is
\[
X=S'\times\R,
\quad
S'=(a_1,b_1)\times\dots\times(a_{N-1},b_{N-1}),
\quad
-\infty\le a_i<b_i\le+\infty,
\]
and the measure is of type~\eqref{eq:measures} with potential
\[
W(x)=-\frac{\|x\|^2}{2}-\sum_{i=1}^{N-1}B_i(x_i),
\]
where each \(B_i\in C^2(a_i,b_i)\) is convex, so that \(\mu\) is finite. 
Clearly, if \(S'=\R^{N-1}\) and \(B_i\equiv0\) for all \(i\), we recover the standard Gaussian measure.

In~\cite[Cor.~4.1]{BCM16}, it is proved that the coordinate half-spaces
\[
H_r(e_N)\cap X=S'\times(r,+\infty),
\qquad
r\in\R,
\]
are the unique isoperimetric sets of their measure, that is,~\eqref{eq:iso-half-spaces} holds with \(\theta=e_N\), the direction carrying the Gaussian factor. 
Actually, this is proved for a more general class of finite measures with continuous, positive, and factorized densities, see~\cite[Thm.~1.1]{BCM16}, to which the following discussion applies verbatim.
In~\cite[Sect.~2]{BCM16}, necessary conditions, of stationarity and stability type, on a density for the coordinate half-spaces to be isoperimetric are also given. 

Up to a rigid motion and a dilation, which rescale measures, perimeters, and Dirichlet energies by constant factors, this class contains two further families of measures, for which the isoperimetric property of half-spaces has also been proved directly.

\smallskip
\emph{Anisotropic Gaussian measures.} Let \(X=\R^N\) and \(W(x)=-\langle Ax,x\rangle/2\), where \(A\) is a symmetric and positive-definite matrix, with eigenvalues ordered as \(\alpha_1\ge\dots\ge\alpha_N>0\). 
By~\cite[Thm.~1.1]{Yeh24},~\eqref{eq:iso-half-spaces} holds for every unit vector \(\theta\) in the eigenspace of \(\alpha_N\), that is, the boundary hyperplanes of the minimizing half-spaces are orthogonal to a direction of largest variance of \(\mu\). 
This also follows from~\cite{BCM16}: in an orthonormal basis of eigenvectors of \(A\) whose last element is \(\theta\), the dilation \(x\mapsto\sqrt{\alpha_N}\,x\) turns \(W\) into
\[
W
=
-\frac{\|x\|^2}{2}-\sum_{i<N}
\left(\frac{\alpha_i}{\alpha_N}-1\right)\,\frac{x_i^2}{2}.
\]

\smallskip
\emph{Log-concave one-dimensional perturbations of the Gaussian measure.} Let \(X=\R^N\), or \(X=\R^{N-1}\times(a,b)\), and \(W(x)=\varphi(x_N)-c\|x\|^2\), where \(c>0\) and \(\varphi\) is concave. 
By~\cite[Thm.~5.10]{Ros14}, if \(X=\R^N\) and \(\varphi\) is smooth, the isoperimetric sets are the vertical half-spaces \(H_r(\theta)\), with \(\langle\theta,e_N\rangle=0\), if \(\varphi\) is not affine, and all the half-spaces if \(\varphi\) is affine, in which case \(\mu\) is, up to a multiplicative constant, a translated Gaussian measure. 
Vertical half-spaces are still isoperimetric if \(\varphi\) is not smooth~\cite[Thm.~5.13]{Ros14}, and, in slabs and half-spaces \(X=\R^{N-1}\times(a,b)\), they are precisely the isoperimetric sets~\cite[Thm.~5.12]{Ros14}. 
If \(\varphi\in C^2\), the isoperimetric property of vertical half-spaces also follows from~\cite{BCM16}: choosing coordinates in which the vertical direction is the first one and \(\theta\) is the last one, the dilation \(x\mapsto\sqrt{2c}\,x\) turns \(W\) into
\[
W=
-\frac{\|x\|^2}{2}-B_1(x_1),
\qquad
\text{where }
B_1(t)
=
-\varphi\left(\frac{t}{\sqrt{2c}}\right),
\]
which is convex. 
For \(X=\R^{N-1}\times(0,+\infty)\) and \(\varphi(t)=k\log t\), with \(k\ge0\), a P\'olya--Szeg\H{o} principle has been proved in~\cite[Thm.~3.1]{BCM08}.

\medskip
In all these cases, the \((\mu,f_\theta)\)-Talenti symmetrization, with \(f_\theta\) as in~\eqref{eq:def-f-halfspaces} and \(\theta\) as above, yields the Faber--Krahn and Saint-Venant inequalities, the optimal sets being the half-spaces \(H_r(\theta)\cap X\). 
For \(p=2\), a weighted P\'olya--Szeg\H{o} principle for the Dirichlet integral, which relies on Talenti's result~\cite{Tal97}, and comparison results for linear degenerate elliptic equations have been derived in~\cite[Sect.~5]{BCM16}. 
Hence, for \(p=2\) and for the measures of~\cite{BCM16}, including anisotropic Gaussian measures and log-concave perturbations with \(\varphi\in C^2\), the Faber--Krahn inequality and the inequalities for the \((2,q)\)-eigenvalues follow from~\cite{BCM16}, and so does the Saint-Venant inequality, choosing a constant datum in the comparison results. 
To the best of our knowledge, the remaining cases, in particular \(p\neq2\), have not been stated before.

\subsection{Isoperimetric sets are complements of concentric balls}
\label{sec:complements}

Let \(X\subset\R^2\) be an open convex sector with vertex at the origin and opening angle \(\theta_0\in(0,\pi]\), so that \(0\in\partial X\), and, for \(k>0\), let
\[
W(x)
=
-(2+k)\log\|x\|,
\qquad
\text{that is,}
\qquad
e^{W(x)}
=
\|x\|^{-2-k},
\qquad 
x\in X.
\]
The potential \(W\) is smooth in \(X\), as required, while the density blows up at the vertex, which lies on \(\partial X\). 
Integrating in polar coordinates, for all \(r>0\), we have
\begin{equation}
\label{eq:complements-vol-per}
\mu(X\setminus\overline{B_r(0)})
=
\frac{\theta_0}{k}\,r^{-k},
\qquad
\Per_\mu(X\setminus\overline{B_r(0)};X)
=
\theta_0\,r^{-k-1},
\end{equation}
whereas \(\mu(B_r(0)\cap X)=+\infty\). 
In particular, \(\mu(X)=+\infty\), and \(\mu\) is not finite on bounded subsets of \(X\).

By~\cite[Remark after Thm.~4.17]{DHHT12}, which extends~\cite[Prop.~4.3]{CJQW08} from the plane to sectors of any angle, the sets \(X\setminus\overline{B_r(0)}\), \(r>0\), are isoperimetric in \(X\), the perimeter not counting \(\partial X\), as in~\eqref{eq:P-mu}. 
Hence, by~\eqref{eq:complements-vol-per}, condition~\ref{hp:T1} holds with
\[
I(m)
=
k^{\frac{k+1}{k}}\,\theta_0^{-\frac1k}\,m^{\frac{k+1}{k}},
\qquad 
m\ge0,
\]
which is continuous on \([0,+\infty)=[0,\mu(X))\) and positive on \((0,+\infty)\). 

We now let
\[
f(x)=\frac{\|x\|^2}{1+\|x\|^2},
\qquad\forall\,x\in X,
\]
that is, \(1-f\) is the function of \cref{sec:balls}. 
As in \cref{sec:balls}, \ref{hp:N1} and~\ref{hp:N2} hold, as \(0\notin X\), the level sets \(\{f=t\}=\partial B_r(0)\cap X\), with \(r=\sqrt{t/(1-t)}\), are connected, so that~\ref{hp:T3} holds, and the lines of steepest ascent are rays emanating from the vertex, so that~\ref{hp:T4} holds. 
Moreover, the superlevel sets \(\{f>t\}=X\setminus\overline{B_r(0)}\) have finite measure by~\eqref{eq:complements-vol-per}, so that~\ref{hp:N3} holds, and they saturate the isoperimetric inequality, so that~\ref{hp:T2} holds as well. 
Therefore, \cref{thm:polya-szego}, \cref{cor:fk-ineq,cor:sv-ineq}, and \cref{rem:pq-eingevalue} hold, among open sets \(\Omega\subset X\) with \(\mu(\Omega)<+\infty\), the optimal sets being \(X\setminus\overline{B_R(0)}\).

We point out that \(\lambda^\mu_{1,p}(X\setminus\overline{B_R(0)})=0\) for all \(R>0\), so that the Faber--Krahn inequality is trivial in this setting, whereas \(T^\mu_p(X\setminus\overline{B_R(0)})\) is finite, and the Saint-Venant inequality is meaningful, if and only if \(k>p/(p-1)\).

In the forthcoming paper~\cite[Sect.~5]{BS-geom}, we extend the isoperimetric inequality above from planar sectors to the products \(X\times(0,1)^\ell\), \(\ell\ge1\), endowed with the same density \(\|y\|^{-2-k}\), which only depends on the variable \(y\in X\), and we show that the complements of the cylinders \(\overline{B_r(0)}\times(0,1)^\ell\), namely the sets \((X\setminus\overline{B_r(0)})\times(0,1)^\ell\), \(r>0\), are isoperimetric. Hence, \cref{thm:polya-szego}, \cref{cor:fk-ineq,cor:sv-ineq}, and \cref{rem:pq-eingevalue} also hold in this setting, the optimal sets being complements of cylinders.

\subsection{Isoperimetric sets known only for small volumes}
\label{sec:small-volumes}

Let $X\subset\mathbb{R}^N$, with \(N\ge 2\), be a bounded open convex set, let $\mu$ be such that \(\mu(X)<+\infty\), and let \(v\in\partial X\). 
Assume that there exists \(\overline{m}\in(0,\mu(X))\) such that, for all \(m\in(0,\overline{m}]\), there exists \(r=r(m)>0\) with \(\mu(B_r(v)\cap X)=m\) and
\begin{equation}
\label{eq:iso-small}
\Per_\mu (B_r(v)\cap X; X) 
\le
\Per_\mu(\Omega; X),
\qquad \forall \Omega\subset X
\text{ with $\mu(\Omega)=m$,}
\end{equation}
that is,~\ref{hp:T1'} holds with
\[
I(m) 
=
\Per_\mu(B_{r(m)}(v)\cap X; X),\quad m\in(0,\overline{m}],
\qquad
I(0)=0,
\]
provided \(I\) is continuous. 
Since \(X\) is bounded, the function of \cref{sec:balls} has positive infimum on \(X\), and we replace it by another decreasing function of \(\|x-v\|\), namely
\[
f(x) 
=
1-\frac{\|x-v\|^2}{D^2}, 
\qquad 
\forall x\in X,
\qquad
\text{where }
D
=
\sup_{x\in X}\|x-v\|.
\]
This does not affect the symmetrization, which only depends on the family of superlevel sets of \(f\). 
Conditions~\ref{hp:N1}--\ref{hp:N3} hold, since \(\inf f=0\) and \(\sup f=1\) are not attained in \(X\), \(\nabla f\neq0\) in \(X\), and \(\mu(X)<+\infty\).
The lines of steepest ascent are segments pointing toward the center \(v\), and the superlevel sets \(\{f>t\}=B_r(v)\cap X\), with \(r=D\sqrt{1-t}\), are isoperimetric for \(t\) close to \(1\). 
Hence,~\ref{hp:T2'} and~\ref{hp:T4} hold, and, provided~\ref{hp:T3'} holds, that is, the level sets \(\{f=t\}=\partial B_r(v)\cap X\) are connected for all \(r\in(0,r(\overline{m}))\), all the assumptions of \cref{rem:threshold} are met, and the Faber--Krahn and Saint-Venant inequalities hold among open sets \(\Omega\subset X\) with \(\mu(\Omega)\le\overline{m}\), the optimal sets being \(B_r(v)\cap X\).

\subsubsection{Convex polytopes}
\label{sec:polytopes}

Let \(X\) be the interior of a convex polytope, endowed with the Lebesgue measure, and let \(v\) be a vertex of \(X\) whose tangent cone has the smallest solid angle.
By~\cite[Thm.~6.8]{RV15},~\eqref{eq:iso-small} holds for some \(\overline{m}>0\).
We claim that, up to reducing \(\overline{m}\), all the assumptions of \cref{rem:threshold} are met. 
Write \(X=\bigcap_{i=1}^kH_i\), where \(H_i\) are open half-spaces, with \(v\in\partial H_i\) for \(i\le j\) and \(v\in H_i\) for \(i>j\), and let
\[
\delta
=
\min_{i>j}\dist(v;\partial H_i)
>
0,
\qquad
C
=
\bigcap_{i\le j}H_i-v,
\qquad
\Sigma
=
C\cap\partial B_1(0),
\]
so that \(C\) is an open convex cone with vertex at the origin, and \(\Sigma\) is connected, being the image of the convex set \(C\) under the continuous map \(y\mapsto y/\|y\|\). 
For \(r\in(0,\delta)\), we have \(B_r(v)\subset H_i\) for all \(i>j\), whence \(B_r(v)\cap X=B_r(v)\cap(v+C)\), so that
\[
\partial B_r(v)\cap X=v+r\,\Sigma
\]
is connected, and
\[
\mathcal{L}^N(B_r(v)\cap X)
=
\frac{|\Sigma|}{N}\,r^N,
\qquad
\Per(B_r(v)\cap X;X)
=
|\Sigma|\,r^{N-1},
\]
where \(|\Sigma|\) is the \(\mathcal{H}^{N-1}\)-measure of \(\Sigma\). 
Therefore, choosing \(\overline{m}\) sufficiently small so that \(r(\overline{m})<\delta\), \ref{hp:T3'} holds, and \(I(m)=N^{\frac{N-1}{N}}|\Sigma|^{\frac1N}m^{\frac{N-1}{N}}\) is continuous on \([0,\overline{m}]\).

Hence, in every convex polytope, the first eigenvalue of the \(p\)-Laplacian on an open set \(\Omega\subset X\) with \(\mathcal{L}^N(\Omega)\le\overline{m}\), with Dirichlet conditions on \(\partial\Omega\cap X\) and Neumann conditions on \(\partial\Omega\cap\partial X\), is not smaller than that of the corner region \(B_r(v)\cap X\) of the same measure, and analogously for the \(p\)-torsional rigidity. 
%

\section{Further developments}
\label{sec:future}
We conclude by describing two further consequences of the framework of this paper, which will appear in forthcoming works.

\subsection{Kohler-Jobin inequality}
\label{ssec:kj-ineq}
In the Euclidean setting of \cref{sec:euclidean}, P\'olya conjectured that the ball minimizes the first eigenvalue of the Dirichlet \(p\)-Laplacian among sets of given \(p\)-torsional rigidity, that is,
\begin{equation}
\label{eq:kj-ineq}
\lambda_{1,p} (B) \le \lambda_{1,p} (\Omega), \qquad \forall \Omega\subset \mathbb{R}^N \text{ with } T_p(\Omega)=T_p(B).
\end{equation}
This inequality is named after Kohler-Jobin, who first proved it for \(p=2\) in~\cite{KJ78,KJ82}, whereas for general \(p\) it was proved much later in~\cite{Bra14}.
It can be thought of as a strengthening of the Faber--Krahn inequality, as the measure constraint is weaker than the torsional one.
Indeed, due to the rescaling properties of the first Dirichlet \(p\)-eigenvalue and of the \(p\)-torsional rigidity, one can relax the constraint in~\eqref{eq:kj-ineq} to \(T_p(\Omega)\le T_p(B)\).
Then, the Faber--Krahn inequality easily follows from the Saint-Venant and Kohler-Jobin inequalities.

The proofs of~\eqref{eq:kj-ineq} given in~\cite{Bra14,KJ78,KJ82} exploit a sophisticated rearrangement, called \emph{Dirichlet transplantation}, which ports to the \emph{anisotropic} setting~\cite[Sect.~6]{Bra14}.
The only other setting where the weighted analog of~\eqref{eq:kj-ineq} has been proved is the standard Gaussian one of \cref{sec:gaussian}.
This has been recently settled in~\cite{HL24} for \(p=2\), among convex sets, by adapting the rearrangement of Kohler-Jobin to the Gauss space with the help of some special functions.
The inequality has then been extended to general \(p\) in~\cite{CFMW26}, with a more general method---still tailored to the Gaussian space though---which lacks the ability to completely characterize cases of the equality.

In the forthcoming note~\cite{BS-KJ}, we show that the rearrangement of Kohler-Jobin can be carried out in the abstract framework of this paper.

\subsection{Quantitative inequalities}
\label{ssec:boosted}
The abstract framework presented here naturally lends itself to stability questions.
Suppose that, in place of the classical (relative) isoperimetric inequality required in~\ref{hp:T1}, one has a quantitative version, \emph{i.e.},
\[
I(\mu(\Omega))\bigl(1+c(\mu(\Omega)) \alpha_\mu^\beta(\Omega)\bigr) \le \Per_\mu(\Omega; X),
\]
for all Borel subsets \(\Omega\subset X\) with \(\mu(\Omega)<\mu(X)\), where \(\alpha_\mu(\,\cdot\,)\) is the \(\mu\)-Fraenkel asymmetry index (with respect to the measure \(\mu\) and the corresponding isoperimetric sets), \(\beta\) a positive constant, and \(c\) a positive function, depending on, at most, \(\mu(\Omega)\).
A natural question is whether one can prove a boosted P\'olya--Szeg\H{o} principle as in~\cite[Lem.~7.10]{BDP17}, which deals with the Euclidean case, and then, following ideas from Hansen--Nadirashvili~\cite{HN94}, a (non-sharp) quantitative Faber--Krahn inequality as in~\cite[Thm.~7.11]{BDP17}.
Apart from the Euclidean case, this route has been successfully followed in the standard Gaussian case~\cite{CCLMP24}.

In the forthcoming paper~\cite{BS-quant}, we positively answer this question, for every \(p\in(1,+\infty)\), showing that the above quantitative isoperimetric inequality yields a boosted P\'olya--Szeg\H{o} principle and, in turn, quantitative Saint-Venant and Faber--Krahn inequalities with exponent \(\beta+1\), up to a nondegeneracy condition on \(I\).
These results apply, in particular, to the convex cones with homogeneous weights of \cref{sec:CROS}, by the quantitative isoperimetric inequality of~\cite{CFPROS22}.


\bibliographystyle{plainurl} 
\bibliography{fk-type-ineq-wittens} 

\end{document}